\documentclass[a4paper,10pt]{article}
\usepackage{amsmath,amsthm,amssymb}
\usepackage[affil-it]{authblk}
\usepackage{comment}
\usepackage{enumerate}

\numberwithin{equation}{section}

\newcommand{\ot}{\otimes}
\newcommand{\id}{\mathord{\operatorname{id}}}

\newcommand{\R}{\mathbb{R}}
\newcommand{\N}{\mathbb{N}}
\newcommand{\C}{\mathbb{C}}

\newcommand{\IrrG}{{\rm Irr}(G)}

\newcommand{\morph}{{\rm Mor}}

\newcommand{\PolG}{{\rm Pol}(G)}

\newcommand{\RepG}{{\rm Rep}(G)}

\newcommand{\irr}{{\rm Irr}}

\theoremstyle{plain}
\newtheorem{theorem}{Theorem}[section]
\newtheorem*{theoremA}{Theorem A}
\newtheorem*{theoremB}{Theorem B}

\newtheorem{lemma}[theorem]{Lemma}
\newtheorem{proposition}[theorem]{Proposition}

\theoremstyle{definition}
\newtheorem{definition}[theorem]{Definition}

\allowdisplaybreaks[4]

\begin{document}

\begin{center}
{\LARGE\bf  Rapid decay and polynomial growth for bicrossed products}
\bigskip

{\sc Pierre Fima and Hua Wang}

\end{center}

\begin{abstract}
  \noindent We study the rapid decay property and polynomial growth
  for duals of bicrossed products coming from a matched pair of a
  discrete group and a compact group.
\end{abstract}

\section{Introduction}

In the breakthrough paper paper \cite{Ha78}, Haagerup showed that the
norm of the reduced C*-algebra $C^*_r(\mathbb{F}_N)$ of the free group
on $N$-generators $\mathbb{F}_N$, can be controlled by the Sobolev
$l^2$-norms associated to the word length function on
$\mathbb{F}_N$. This is a striking phenomenon which actually occurs in
many more cases. Jolissaint recognized this phenomenon, called Rapid
Decay (or property $(RD)$), and studied it in a systematic way in
\cite{Jo90}. Property $(RD)$ has now many applications. Let us mention
the remarkable one concerning K-theory. Property $(RD)$ allowed
Jolissaint \cite{Jo89} to show that the K-theory and $C^*_r(\Gamma)$ equals
the K-theory of subalgebras of rapidly decreasing functions on
$\Gamma$ (Jolissaint did attribute this result to Connes). This result was then used by V. Lafforgue in his approach to the
Baum-Connes conjecture via Banach KK-theory \cite{La00, La02}.

In this paper, we view discrete quantum groups as duals of compact
quantum groups. The theory of compact quantum groups has been
developed by Woronowicz \cite{Wo87, Wo88, Wo98}. Property $(RD)$ for
discrete quantum groups has been introduced and studied by Vergnioux
\cite{Ve07}. Property $(RD)$ has been refined later \cite{BVZ14} in
order to fit in the context of non-unimodular discrete quantum groups.

In this paper, we study the permanence of property $(RD)$ under the
bicrossed product construction. This construction was initiated by Kac
\cite{Ka68} in the context of finite quantum groups and was
extensively studied later by many authors in different settings. The
general construction, for locally compact quantum groups, was
developed by Vaes-Vainerman \cite{VV03}. In the context of compact
quantum groups given by matched pairs of classical groups, an easier
approach, that we will follow, was given by Fima-Mukherjee-Patri
\cite{FMP16}.

Following \cite{FMP16}, the bicrossed product construction associates
to a matched pair $(\Gamma,G)$ of a discrete group $\Gamma$ and a compact group
$G$ (see Section \ref{SectionBCP}) a compact quantum group
$\mathbb{G}$, called the bicrossed product. Given a length function
$l$ on the set of equivalence classes $\irr(\mathbb{G})$ of
irreducible unitary representations of $\mathbb{G}$ one can associate
in a canonical way, as explained in Proposition \ref{Proplength}, a
pair of length functions $(l_\Gamma,l_G)$ on $\Gamma$ and $\irr(G)$
respectively. Such a pair satisfies some compatibility relations and
every pair of length functions $(l_\Gamma,l_G)$ on $(\Gamma,\irr(G))$ satisfying
those compatibility relations will be called matched (see Definition
\ref{matchedlength}). Any matched pair $(l_\Gamma,l_G)$ on
$(\Gamma,\irr(G))$ allows one to reconstruct a canonical length function on
$\irr(\mathbb{G})$. The main result of the present paper is the
following.

\begin{theoremA}
  Let $(\Gamma,G)$ be a matched pair of a discrete group $\Gamma$ and a compact
  group $G$. Denote by $\mathbb{G}$ the bicrossed product. The
  following are equivalent.
\begin{enumerate}
\item $\widehat{\mathbb{G}}$ has property $(RD)$.
\item There exists a matched pair of length function $(l_\Gamma,l_G)$ on
  $(\Gamma,\irr(G))$ such that both $(\Gamma,l_\Gamma)$ and
  $(\widehat{G},l_G)$ have (RD).
\end{enumerate}
\end{theoremA}

For amenable discrete groups, property $(RD)$ is equivalent to
polynomial growth \cite{Jo90} and the same occurs for discrete quantum
groups \cite{Ve07}. Hence, for the compact classical group $G$ one has
that $(\widehat{G},l_G)$ has $(RD)$ if and only if it has polynomial
growth. Note that a bicrossed product of a matched pair $(\Gamma,G)$ is
co-amenable if and only if $\Gamma$ is amenable \cite{FMP16}. The following
theorem shows the permanence of polynomial growth under the bicrossed
product construction.

\begin{theoremB}
  Let $(\Gamma,G)$ be a matched pair of a discrete group $\Gamma$ and a compact
  group $G$. Denote by $\mathbb{G}$ the bicrossed product. The
  following are equivalent.
\begin{enumerate}
\item $\widehat{\mathbb{G}}$ has polynomial growth.
\item There exists a matched pair of length function $(l_\Gamma,l_G)$ on
  $(\Gamma,\irr(G))$ such that both $(\Gamma,l_\Gamma)$ and
  $(\widehat{G},l_G)$ have polynomial growth.
\end{enumerate}
\end{theoremB}

The main ingredient to prove Theorem $A$ and $B$ is the classification
of the irreducible unitary representation of a bicrossed product and
the fusion rules.

The paper is organized as follows. Section $2$ is a preliminary
section in which we introduce our notations. In section $3$ we
classify the irreducible unitary representation of a bicrossed product
and describe their fusion rules. Finally, in section $4$, we prove
Theorem A and Theorem B.

\section{Preliminaries}

\subsection{Notations}

For a Hilbert space $H$, we denote by $\mathcal{U}(H)$ its unitary
group and by $\mathcal{B}(H)$ the C*-algebra of bounded linear operators on
$H$. When $H$ is finite dimensional, we denote by ${\rm Tr}$ the
unique trace on $\mathcal{B}(H)$ such that ${\rm Tr}(1)={\rm dim}(H)$. We use
the same symbol $\ot$ for the tensor product of Hilbert spaces,
unitary representations of compact quantum groups, minimal tensor
product of C*-algebras. For a compact quantum group $G$, we denote by
$\IrrG$ the set of equivalence classes of irreducible unitary
representations and $\RepG$ the collection of finite dimensional
unitary representations. We will often denote by $[u]$ the equivalence
class of an irreducible unitary representation $u$. For
$u\in{\rm Rep}(G)$, we denote by $\chi(u)$ its character, i.e., viewing
$u\in\mathcal{B}(H)\ot C(G)$ for some finite dimensional Hilbert space
$H$, one has $\chi(u):=({\rm Tr}\ot\id)(u)\in C(G)$. We denote by
$\PolG$ the unital C*-algebra obtained by taking the Span of the
coefficients of irreducible unitary representation, by $C_m(G)$ the
enveloping C*-algebra of $\PolG$ and by $C(G)$ the C*-algebra
generated by the GNS construction of the Haar state on $C_m(G)$. We
also denote by $\varepsilon\,:\,C_m(G)\rightarrow\C$ the counit and we use the same symbol
$\varepsilon\in\IrrG$ to denote the trivial representation and its class in
${\rm Irr}(G)$. In the entire paper, the word representation means a
unitary and finite dimensional representation.

\subsection{Compact bicrossed products}\label{SectionBCP}

In this section, we follow the approach and the notations of
\cite{FMP16}.

Let $(\Gamma,G)$ be a pair of a countable discrete group $\Gamma$ and a second
countable compact group $G$ with a left action
$\alpha\,:\,\Gamma\rightarrow{\rm Homeo}(G)$ of $\Gamma$ on the compact space
$G$ by homeomorphisms and a right action
$\beta\,:\,G\rightarrow S(\Gamma)$ of $G$ on the discrete space
$\Gamma$, where $S(\Gamma)$ is the Polish group of bijections of
$\Gamma$, the topology being the one of pointwise convergence i.e., the
smallest one for which the evaluation maps
$S(\Gamma)\rightarrow\Gamma$, $\sigma\mapsto\sigma(\gamma)$ are continuous, for all
$\gamma\in\Gamma$, where $\Gamma$ has the discrete topology. Here,
$\alpha$ is a group homomorphism and $\beta$ is an antihomomorphism. The pair
$(\Gamma,G)$ is called a matched pair if $\Gamma \cap G = \{e\}$ with
$e$ being the common unit for both \( G \) and \( \Gamma \), and if the
actions $\alpha$ and $\beta$ satisfy the following matched pair relations:
\begin{equation}\label{EqMatched}
  \forall g, h \in G, \, \gamma, \mu \in \Gamma, \quad
  \alpha_\gamma(gh)=\alpha_\gamma(g)\alpha_{\beta_g(\gamma)}(h),\,\,\beta_g(\gamma\mu)=\beta_{\alpha_s(g)}(\gamma)\beta_g(\mu)\quad\text{and}\quad\alpha_\gamma(e)=
  \beta_g(e)=e.
\end{equation}
We also write $\gamma\cdot g:=\beta_g(\gamma)$. From now on, we assume
\( (\Gamma, G) \) is matched. It is shown in \cite[Proposition 3.2]{FMP16} that
\( \beta \) is automatically continuous. By continuity of $\beta$ and
compactness of $G$, every $\beta$ orbit is finite. Moreover, the sets
$G_{r,s}:=\{g\in G\,:\,r\cdot g=s\}$ are clopen (see \cite[Section 2.1]{FMP16}). Let
$v_{rs}=1_{G_{r,s}}\in C(G)$ be the characteristic function of
$G_{r,s}$. It is shown in \cite[Section 2.1]{FMP16} that, for all $\beta$-orbits
$\gamma \cdot G \in \Gamma / G$, the unitary
$v_{\gamma \cdot G}:=\sum_{r,s\in\gamma\cdot G}e_{rs}\ot v_{rs}\in\mathcal{B}(l^2(\gamma\cdot G))\ot C(G)$ is a
unitary representation of $G$ as well as a magic unitary, where
$e_{rs}\in\mathcal{B}(l^2(\gamma\cdot G))$ are the canonical matrix units and the Haar
probability measure $\nu$ on $G$ is $\alpha$-invariant.

It is shown in \cite[Theorem 3.4]{FMP16} that there exists a unique compact quantum
group $\mathbb{G}$, called the bicrossed product of the matched pair
$(\Gamma,G)$, such that $C(\mathbb{G})=\Gamma{}_\alpha\ltimes C(G)$ is the reduced
C*-algebraic crossed product, generated by a copy of $C(G)$ and the
unitaries $u_\gamma$, $\gamma\in\Gamma$ and
$\Delta\,:\, C(\mathbb{G})\rightarrow C(\mathbb{G})\ot C(\mathbb{G})$ is the unique
unital $*$-homomorphism satisfying $\Delta\vert_{C(G)}=\Delta_G$ (the
comultiplication on $C(G)$) and
$\Delta(u_\gamma)=\sum_{r\in\gamma\cdot G}u_\gamma v_{\gamma r}\ot u_r$ for all
$\gamma\in\Gamma$. It is also shown that the Haar state on
$\mathbb{G}$ is a trace and is given by the formula
$h(u_\gamma F)=\delta_{\gamma,1}\int_GFd\nu$ for all $\gamma\in\Gamma$ and $F\in
C(G)$. 

\section{Representation theory of bicrossed products}

\subsection{Classification of irreducible representations}
In this section we classify the irreducible representations of a
bicrossed product. Let $(\Gamma,G)$ be a matched pair of a discrete
countable group $\Gamma$ and a second countable compact group $G$ with
actions $\alpha$, $\beta$.

For $\gamma\in\Gamma$ we denote by $G_\gamma:= G_{\gamma,\gamma}$ the stabilizer of
$\gamma$ for the action $\beta\,:\,\Gamma\curvearrowleft G$. Note that
$G_\gamma$ is an open (hence closed) subgroup of $G$, hence of finite
index: its index is $\vert\gamma\cdot G\vert$. We view
$C(G_\gamma)=v_{\gamma\gamma}C(G)\subset C(G)$ as a non-unital C*-subalgebra. Let us denote
by $\nu$ the Haar probability measure on $G$ and note that
$\nu(G_\gamma)=\frac{1}{\vert\gamma \cdot G\vert}$ so that the Haar probability measure
$\nu_{\gamma}$ on $G_\gamma$ is given by
$\nu_\gamma(A)=\vert\gamma\cdot G\vert\,\nu(A)$ for all Borel subset $A$ of $G_\gamma$.

For $\gamma\in\Gamma$ we fix a section, still denoted $\gamma$,
$\gamma\,:\,\gamma\cdot G\rightarrow G$ of the canonical surjection
$G\rightarrow\gamma\cdot G\,:\,g\mapsto\gamma\cdot g$. This means that
$\gamma\,:\,\gamma\cdot G\rightarrow G$ is an injective map such that
$\gamma\cdot\gamma(r)=r$ for all $r\in\gamma\cdot G$. We choose the section
$\gamma$ such that $\gamma(\gamma)=1$, for all $\gamma\in\Gamma$. For
$r,s\in\gamma\cdot G$, we denote by $\psi^\gamma_{r,s}$ the
$\nu$-preserving homeomorphism of $G$ defined by
$\psi^\gamma_{r,s}(g)=\gamma(r)g\gamma(s)^{-1}$. It follows from our choices that
$\psi^\gamma_{\gamma,\gamma}=\id$ for all $\gamma\in\Gamma$. Moreover, for all
$g\in G$, one has $\psi^\gamma_{r,s}(g)\in G_\gamma$ if and only if
$g \in G_{r,s}$. It follows that $\psi^\gamma_{r,r}$ is an isomorphism and an
homeomorphism from $G_r$ to $G_\gamma$ intertwining the Haar probability
measures.

Let $u\,:\,G_\gamma\rightarrow\mathcal{U}(H)$ be a unitary representation of
$G_\gamma$ and view $u$ as a continuous function
$G\rightarrow\mathcal{B}(H)$ which is zero outside $G_\gamma$ i.e. a partial isometry in
$\mathcal{B}(H)\ot C(G)$ such that
$uu^*=u^*u=\id_H\ot v_{\gamma\gamma}$. Define, for
$r,s\in\gamma\cdot G$, the partial isometry
$u_{r,s}:=u\circ\psi^\gamma_{r,s}:=(g\mapsto u(\psi^\gamma_{r,s}(g)))\in \mathcal{B}(H)\ot C(G)$ and note
that
$u_{r,s}^*u_{r,s}=u_{r,s}u_{r,s}^*=\id_H\ot
1_{G_{r,s}}$. In the sequel we view
$u_{r,s}\in \mathcal{B}(H)\ot C(G)\subset\mathcal{B}(H)\ot C(\mathbb{G})$ and we define:
$$\gamma(u):=\sum_{r,s\in\gamma\cdot G}e_{rs}\ot(1\ot u_rv_{rs})u_{r,s}\in\mathcal{B}(l^2(\gamma\cdot G))\ot \mathcal{B}(H)\ot C(\mathbb{G}),$$
where we recall that $e_{rs}$, for $r,s\in\gamma\cdot G$, are the matrix units
associated to the canonical orthonormal basis of $l^2(\gamma\cdot G)$.

The irreducible unitary representations of $\mathbb{G}$ are described
as follows.

\begin{theorem}\label{IrrBC}
The following holds.
\begin{enumerate}
\item For all $\gamma\in\Gamma$ and $u\in{\rm Rep}(G_\gamma)$ one has $\gamma(u)\in{\rm Rep}(\mathbb{G})$.
\item The character of $\gamma(u)$ is $\chi(\gamma(u))=\sum_{r\in\gamma\cdot G}u_rv_{rr}\chi(u)\circ\psi^\gamma_{r,r}$.
\item For all $\gamma,\mu\in\Gamma$, $u\in{\rm Rep}(G_\gamma)$ and $w\in{\rm Rep}(G_\mu)$ one has
  $${\rm dim}({\rm Mor}_\mathbb{G}(\gamma(u),\mu(w)))=\delta_{\gamma\cdot G,\mu\cdot G}{\rm
    dim}({\rm Mor}_{G_\gamma}(u,w\circ\psi_{\gamma,\gamma}^\mu)).$$
\item For all $\gamma\in\Gamma$ and $u\in{\rm Rep}(G_\gamma)$ one has
  $\overline{\gamma(u)}\simeq\gamma^{-1}(\overline{u}\circ\alpha_{\gamma^{-1}})$ (which makes sense
  since
  $\alpha_{\gamma^{-1}}\,:\, G_{\gamma^{-1}}\rightarrow G_\gamma$ is a group isomorphism and an
  homeomorphism).
\item $\gamma(u)$ is irreducible if and only if $u$ is
  irreducible. Moreover, for any irreducible unitary representation
  $u$ of $\mathbb{G}$ there exists $\gamma\in\Gamma$ and $v$ an irreducible
  representation of $G_\gamma$ such that $u\simeq\gamma(v)$.
\end{enumerate}
\end{theorem}

\begin{proof}
  $(1)$. Writing $\gamma(u)=\sum_{r,s}e_{r,s}\ot V_{r,s}$, where
  $V_{r,s}:=(1\ot u_rv_{rs})u_{r,s}\in\mathcal{B}(H)\ot C(\mathbb{G})$, it
  suffices to check that, for all $r,s\in\gamma\cdot G$ one has
  $(\id\ot\Delta)(V_{r,s})=\sum_{t\in\gamma\cdot G}(V_{r,t})_{12}(V_{t,s})_{13}$. We
  first claim that, for all $r,s\in\gamma\cdot G$,
  $(\id\ot\Delta)(u_{r,s})=\sum_{t\in\gamma\cdot G}(u_{r,t})_{12}(u_{t,s})_{13}$. To
  check our claim, first recall that, for all $r,s\in\gamma\cdot G$ one has
  $\psi^\gamma_{r,s}(g)\in G_\gamma$ if and only if $r\cdot g=s$. Let
  $r,s\in\gamma\cdot G$ and $g,h\in G$. For $t=r\cdot g\in\gamma\cdot G$ one has :
$$u_{r,s}(gh)=u(\gamma(r)g\gamma(t)^{-1}\gamma(t)h\gamma(s)^{-1})=u(\psi^\gamma_{r,t}(g)\psi^\gamma_{t,s}(h))=\left\{\begin{array}{ll}u_{r,t}(g)u_{t,s}(h)&\text{if }r\cdot gh=s,\\0&\text{otherwise.}\end{array}\right.$$
Since we also have $u_{t,s}(h)=0$ whenever $r\cdot gh\neq s$ we find, in both
cases, that $u_{r,s}(gh)=u_{r,t}(g)u_{t,s}(h)$. Now, for
$t\neq r\cdot g$ we have $u_{r,t}(g)=0$ so the following formulae holds for
any $r,s\in \gamma\cdot G$ and any $g,h\in G$:
$$v_{r,t}(g)u_{r,s}(gh)=u_{r,t}(g)u_{t,s}(h).$$
Hence, for all $r,s,t\in\gamma\cdot G$,
$(1\ot v_{r,t}\ot
1)(\id\ot\Delta)(u_{r,s})=(u_{r,t})_{12}(u_{t,s})_{13}$. Using this we
find:
\begin{eqnarray*}
  \sum_{t\in\gamma\cdot G}(V_{r,t})_{12}(V_{t,s})_{13}&=&\sum_t(1\ot u_rv_{rt}\ot 1)(u_{r,t})_{12}(1\ot 1\ot u_tv_{ts})(u_{t,s})_{13}\\
                                        &=&\sum_t(1\ot u_rv_{rt}\ot u_tv_{ts})(u_{r,t})_{12}(u_{t,s})_{13}=\left(1\ot(\sum_tu_rv_{rt}\ot u_tv_{ts})\right)(\id\ot\Delta)(u_{r,s}).
\end{eqnarray*}
Since $v_\gamma$ is a unitary representation of $G$ and a magic unitary we
also have:
$$\Delta(u_rv_{rs})=\sum_{t,t'}(u_rv_{rt}\ot u_t)(v_{rt'}\ot v_{t's})=\sum_tu_rv_{rt}\ot u_tv_{ts}.$$
This shows that $\gamma(u)$ is a representation of $\mathbb{G}$. We now
check that $\gamma(u)$ is unitary. As before, since for all
$r,s\in\gamma\cdot G$ one has $\psi^\gamma_{r,s}(g)\in G_\gamma$ if and only if
$r\cdot g=s$ and because $u$ is a unitary representation of
$G_\gamma$, we have, for all $r,t\in\gamma\cdot G$,
$(1\ot v_{rt})u_{r,t}u_{r,t}^*=1\ot v_{rt}$. Hence,
\begin{eqnarray*}
  \sum_{t\in\gamma\cdot G}V_{r,t}V_{s,t}^*&=&\sum_t(1\ot u_r)(1\ot v_{rt})u_{r,t}u_{s,t}^*(1\ot v_{st})(1\ot u_s^*)\\
                            &=&\delta_{r,s}(1\ot u_r)\left(\sum_t(1\ot v_{rt})u_{r,t}u_{r,t}^*\right)(1\ot u_r^*)=\delta_{r,s}(1\ot u_r)\left(\sum_t(1\ot v_{rt})\right)(1\ot u_r^*)\\
                            &=&\delta_{r,s}.
\end{eqnarray*}
A similar computations shows that
$\sum_{t\in\gamma\cdot G}V_{t,r}^*V_{t,s}=\delta_{r,s}$.

$(2)$. The character of $\gamma(u)$ is given by
$$\chi(\gamma(u))=\sum_{r\in\gamma\cdot G}({\rm Tr}\ot\id)(V_{r,r})=\sum_ru_rv_{rr}({\rm Tr}\ot\id)(u_{r,r})=\sum_ru_rv_{rr}\chi(u)\circ\psi^\gamma_{r,r}.$$

$(3)$. Let $\gamma,\mu\in\Gamma$ and $u,w$ be representations of
$G_\gamma$ and $G_\mu$ respectively. Since the Haar measure on $G$ is
invariant under the action $\alpha$ and the homeomorphisms
$\psi^\gamma_{r,r}$ and $\psi^\mu_{r,r}$, we find, by the character formulae in $2$ and the crossed-product relations,
\begin{eqnarray*}
  {\rm dim}({\rm Mor}(\gamma(u),\mu(w)))&=&h(\chi(\gamma(u))\chi(\mu(w))^*)=\sum_{r\in\gamma\cdot G,\,s\in\mu\cdot G}h(u_{rs^{-1}}\alpha_s(v_{rr}v_{ss}\chi(u)\circ\psi^\gamma_{r,r}(\chi(w)\circ\psi^\mu_{s,s})^*))\\
                                 &=&\delta_{\gamma\cdot G,\mu\cdot G}\sum_{r\in\gamma\cdot G}\int_G\alpha_r(v_{rr}(\chi(u)\circ\psi^\gamma_{r,r})(\overline{\chi(w)}\circ\psi^\mu_{r,r}))d\nu\\
                                 &=&\delta_{\gamma\cdot G,\mu\cdot G}\sum_{r\in\gamma\cdot G}\int_{G_r}(\chi(u)\circ\psi^\gamma_{r,r})(\chi(\overline{w})\circ\psi^\mu_{r,r})d\nu\\
                                 &=&\delta_{\gamma\cdot G,\mu\cdot G}\sum_{r\in\gamma\cdot G}\int_{G_\mu}\chi(u)\circ(\psi^\mu_{\gamma,\gamma})^{-1}(\chi(\overline{w})\circ\psi^\mu_{r,r}\circ(\psi^\gamma_{r,r})^{-1}\circ(\psi^\mu_{\gamma,\gamma})^{-1})d\nu\\
\end{eqnarray*}
Now, note that
$\psi^\mu_{r,r}\circ(\psi^\gamma_{r,r})^{-1}\circ(\psi^\mu_{\gamma,\gamma})^{-1}={\rm Ad}(h)$, where
$h=\mu(r)\gamma(r)^{-1}\mu(\gamma)^{-1}$. Moreover,
$\mu\cdot h=\mu$ since:
$$\mu\cdot \mu(r)\gamma(r)^{-1}\mu(\gamma)^{-1}=r\cdot\gamma(r)^{-1}\mu(\gamma)^{-1}=\gamma\cdot\mu(\gamma)^{-1}=\mu.$$
Hence, $h\in G_\mu$. Since the characters of finite dimensional
unitary representation of a group $\Lambda$ are central functions i.e.
invariant under ${\rm Ad}(\lambda)$ for all $\lambda\in \Lambda$, we have
$\chi(\overline{w})\circ\psi^\mu_{r,r}\circ(\psi^\gamma_{r,r})^{-1}\circ(\psi^\mu_{\gamma,\gamma})^{-1}=\chi(\overline{w})\circ{\rm
  Ad}(h)=\chi(\overline{w})$. Hence:
\begin{eqnarray*}
{\rm dim}({\rm Mor}(\gamma(u),\mu(w)))&=&\delta_{\gamma\cdot G,\mu\cdot G}\sum_{r\in\gamma\cdot G}\int_{G_\mu}\chi(u)\circ(\psi^\mu_{\gamma,\gamma})^{-1}\chi(\overline{w})d\nu
=\delta_{\gamma\cdot G,\mu\cdot G}\int_{G_\mu}\chi(u)\circ(\psi^\mu_{\gamma,\gamma})^{-1}\chi(\overline{w})d\nu_\mu\\
&=&\delta_{\gamma\cdot G,\mu\cdot G}{\rm dim}({\rm Mor}_{G_\mu}(u\circ(\psi_{\gamma,\gamma}^\mu)^{-1},w))=\delta_{\gamma\cdot G,\mu\cdot G}\int_{G_\gamma}\chi(u)\chi(\overline{w}\circ\psi^\mu_{\gamma,\gamma})d\nu_\mu\\
&=&\delta_{\gamma\cdot G,\mu\cdot G}{\rm dim}({\rm Mor}_{G_\gamma}(u,w\circ\psi_{\gamma,\gamma}^\mu).
\end{eqnarray*}

$(4)$. Note that, by the bicrossed product relations, we have, for all
$\gamma\in\Gamma$ and $g\in G$,
$(\gamma\cdot g)^{-1}=\gamma^{-1}\cdot\alpha_{\gamma}(g)$. Hence
$v_{\gamma^{-1}\gamma^{-1}}\circ\alpha_\gamma=v_{\gamma\gamma}$ and
$(\gamma\cdot G)^{-1}=\gamma^{-1}\cdot G$. In particular,
$\alpha_{\gamma}\,:\, G_{\gamma}\rightarrow G_{\gamma^{-1}}$ is an homeomorphism and, by the
bicrossed product relations, one has, for all $g\in G_\gamma$ and
$h\in G$,
$\alpha_\gamma(gh)=\alpha_\gamma(g)\alpha_{\gamma\cdot g}(h)=\alpha_\gamma(g)\alpha_\gamma(h)$ so that
$\alpha_\gamma\,:\, G_{\gamma}\rightarrow G_{\gamma^{-1}}$ is also a group homomorphism.

For $r\in\gamma\cdot G$ one has
$\gamma^{-1}\cdot\alpha_\gamma(\gamma(r))=(\gamma\cdot\gamma(r))^{-1}=r^{-1}=\gamma^{-1}\cdot\gamma^{-1}(r^{-1})$. This
implies that, for all $\gamma\in\Gamma$, there exists a map
$\eta_\gamma\,:\,\gamma\cdot G\rightarrow G_{\gamma^{-1}}$ such that, for all
$r\in\gamma\cdot G$, one has $\alpha_\gamma(\gamma(r))=\eta_\gamma(r)\gamma^{-1}(r^{-1})$.

Let now $r\in\gamma\cdot G$ and $g\in G_r$. One has, using the bicrossed product
relations, that
$e=\alpha_{r}(\gamma(r){\gamma(r)}^{-1})=\alpha_{\gamma}(\gamma(r))\alpha_{r}({\gamma(r)}^{-1})$, hence
$$(\alpha_\gamma\circ\psi^\gamma_{r,r})(g)=\alpha_\gamma(\gamma(r))\alpha_r(g)\alpha_r(\gamma(r)^{-1})=\alpha_\gamma(\gamma(r))\alpha_r(g)\bigl(\alpha_\gamma(\gamma(r)))\bigr)^{-1}=\eta_\gamma(r)(\psi^{\gamma^{-1}}_{r^{-1},r^{-1}}\circ\alpha_r)(g)(\eta_\gamma(r))^{-1}.$$
Hence, for all $\gamma\in\Gamma$, if $w\in{\rm Rep}(G_{\gamma^{-1}})$, since
$\chi(w)\in C(G_{\gamma^{-1}})$ is central we have
$$\chi(w)\circ\alpha_\gamma\circ\psi^\gamma_{r,r}(g)=\chi(w)\circ\psi^{\gamma^{-1}}_{r^{-1},r^{-1}}\circ\alpha_r(g)\quad\text{for all  }r\in\gamma\cdot G, \, g\in G_r.$$

Since, as we seen above, $\gamma^{-1}\cdot G=(\gamma\cdot G)^{-1}$ and because $\chi(\overline{u}\circ\alpha_{\gamma^{-1}})=\chi(\overline{u})\circ\alpha_{\gamma^{-1}}$ we find, by the character formulae in $2$, $\chi(\gamma^{-1}(\overline{u}\circ\alpha_{\gamma^{-1}}))=\sum_{r\in\gamma\cdot G}u_{r^{-1}}v_{r^{-1}r^{-1}}\chi(\overline{u})\circ\alpha_{\gamma^{-1}}\circ\psi^{\gamma^{-1}}_{r^{-1},r^{-1}}$. It then follows from the crossed-product relations and the discussion above :
\begin{eqnarray*}
\chi(\gamma^{-1}(\overline{u}\circ\alpha_{\gamma^{-1}}))&=&\sum_{r\in\gamma\cdot G}u_{r^{-1}}v_{r^{-1}r^{-1}}\chi(\overline{u})\circ\alpha_{\gamma^{-1}}\circ\psi^{\gamma^{-1}}_{r^{-1},r^{-1}}\\
&=&\sum_{r\in\gamma\cdot G}(\chi(\overline{u})\circ\alpha_{\gamma^{-1}}\circ\psi^{\gamma^{-1}}_{r^{-1},r^{-1}}\circ\alpha_r)(v_{r^{-1}r^{-1}}\circ\alpha_r )u_{r^{-1}}\\
&=&\sum_{r\in\gamma\cdot G}\chi(\overline{u})\circ\psi^{\gamma}_{r,r}v_{rr}u_{r}^*=\sum_{r\in\gamma\cdot G}(\chi(u)\circ\psi^{\gamma}_{r,r}v_{rr})^*u_{r}^*\\
&=&\chi(\gamma(u))^*
\end{eqnarray*}

$(5)$. The statement on irreducibility following from $3$, it suffices, by the general theory, to show that the linear span
$X$ of coefficients of representations of the form $\gamma(u)$, for
$\gamma\in\Gamma$ and $u$ an irreducible unitary representation of
$G_\gamma$, is a dense subset of $C(\mathbb{G})$. Note that, for all
$\gamma\in\Gamma$, the relation
$1=\sum_{r\in\gamma\cdot G}v_{\gamma r}$ implies that any function in
$C(G)$ is a sum of continuous functions with support in
$G_{\gamma,r}:=\{g\in G\,:\,\gamma\cdot g=r\}$, for
$r\in\gamma\cdot G$. Moreover, since
$G_{\gamma,r}=(\psi^\gamma_{\gamma,r})^{-1}(G_\gamma)$, any continuous function on
$G$ with support in $G_{\gamma, r}$ is of the form
$ F\circ\psi^\gamma_{\gamma,r}$, where $F\in C(G_\gamma)$. Since the linear span of
coefficients of irreducible unitary representation of $G_\gamma$ is dense
in $C(G_\gamma)$, it suffices to show that, for any $\gamma\in\Gamma$, for any
irreducible unitary representation of $G_\gamma$,
$u\,:\,G_\gamma\rightarrow\mathcal{U}(H)$, any coefficient
$u_{ij}\in C(G_\gamma)=v_{\gamma \gamma}C(G)\subset C(G)$ satisfies
$u_\gamma u_{ij}\in X$. But this is obvious since one has
\begin{displaymath}
  u_\gamma u_{ij}=u_\gamma v_{\gamma\gamma}u_{i,j}=u_\gamma
v_{\gamma\gamma}u_{i,j}\circ\psi^\gamma_{\gamma,\gamma}=\gamma(u)_{\gamma,\gamma,i,j}\in X. \qedhere
\end{displaymath}
\end{proof}

Finally, the fusion rules are described as follows.

Let $\gamma,\mu\in\Gamma$,
$u\,:\, G_\gamma\rightarrow\mathcal{U}(H_u)$,
$v\,:\, G_\mu\rightarrow\mathcal{U}(H_v)$ by unitary representations of
$G_\gamma$ and $G_\mu$ respectively. For any
$r\in(\gamma\cdot G)(\mu\cdot G)$, we define the $r$-twisted tensor product of
$u$ and $v$, denoted $u\underset{r}{\ot} v$ as a unitary
representation of $G_r$ on $K_r\ot H_u\ot H_v$, where
$$K_r:={\rm Span}(\{e_s\ot e_t\,:\,s\in\gamma\cdot G\text{ and }t\in\mu\cdot G\text{ such that }st=r\})\subset l^2(\gamma\cdot G)\ot l^2(\mu\cdot G).$$
For $g\in G$, we define:
$$(u\underset{r}{\ot}v)(g)=\sum_{\substack{s,s'\in\gamma\cdot G\\t,t'\in\mu\cdot G\\st=r=s't'}}e_{ss'}\ot e_{tt'}\ot v_{ss'}(\alpha_t(g))v_{tt'}(g)u(\psi^\gamma_{s,s'}(\alpha_t(g)))\ot v(\psi^\mu_{t,t'}(g))\in\mathcal{U}(K_r\ot H_u\ot H_v).$$

\begin{theorem}\label{ThmFusion}
The following holds.
\begin{enumerate}
\item For all $\gamma,\mu\in\Gamma$, all
  $r\in(\gamma\cdot G)(\mu\cdot G)$ and all $u,v$ finite dimensional unitary
  representations of $G_\gamma$, $G_\mu$ respectively the element
  $u\underset{r}{\ot} v$ is a unitary representation of $G_r$.
\item The character of $u\underset{r}{\ot} w$ is
  $\chi(u\underset{r}{\ot} v)=\sum_{s\in\gamma\cdot G,t\in\mu\cdot
    G,\,st=r}(v_{ss}\circ\alpha_t)v_{tt}(\chi(u)\circ\psi^\gamma_{s,s}\circ\alpha_t)(\chi(v)\circ\psi^\mu_{t,t}).$
\item For all $\gamma_1,\gamma_2,\gamma_3\in\Gamma$ and all $u,v,w$ unitary representations
  of $G_{\gamma_1}$, $G_{\gamma_2}$ and $G_{\gamma_3}$ respectively, the number
  $\dim(\morph_{\mathbb{G}}(\gamma_1(u),\gamma_2(v)\ot\gamma_3(w)))$ is equal to:
$$\left\{\begin{array}{ll}\frac{1}{\vert \gamma_1\cdot G\vert}\sum_{r\in\gamma_1\cdot G\cap(\gamma_2\cdot G)(\gamma_3\cdot G)}{\rm dim}(\morph_{G_r}(u\circ\psi^{\gamma_1}_{r,r},v\underset{r}{\ot}w))&\text{if }\gamma_1\cdot G\cap(\gamma_2\cdot G)(\gamma_3\cdot G)\neq\emptyset,\\0&\text{otherwise.}\end{array}\right.$$
\end{enumerate}
\end{theorem}
Let us observe that, by the bicrossed product relations, we have, for
all $\gamma_1,\gamma_2,\gamma_3\in\Gamma$,
$$\gamma_1\cdot G\cap(\gamma_2\cdot G)(\gamma_3\cdot G)\neq\emptyset\Leftrightarrow\gamma_1\cdot G\subset(\gamma_2\cdot G)(\gamma_3\cdot G).$$
\begin{proof}
$(1)$. Put $w=u\underset{r}{\ot} v$ and let $g,h\in G_r$. Then, $w(gh)$ is equal to:
$$\sum_{s,s'\in\gamma\cdot G,t,t'\in\mu\cdot G,st=s't'=r}e_{ss'}\ot e_{tt'}\ot v_{ss'}(\alpha_t(gh))v_{tt'}(gh)u(\psi^\gamma_{s,s'}(\alpha_t(gh)))\ot v(\psi^\mu_{t,t'}(gh)).$$
Since $v_{ty}(g)\ne0$ precisely when $t\cdot g=y$, the factor
$v_{ss'}(\alpha_t(gh))v_{tt'}(gh)u(\psi^\gamma_{s,s'}(\alpha_t(gh)))\ot
v(\psi^\mu_{t,t'}(gh))$ is equal to:
$$\sum_{x\in\gamma\cdot G,y\in\mu\cdot G}v_{sx}(\alpha_t(g))v_{xs'}(\alpha_{t\cdot g}(h))v_{ty}(g)v_{yt'}(h)u(\psi^\gamma_{s,x}(\alpha_t(g))u(\psi^\gamma_{x,s'}(\alpha_{t\cdot g}(h)))\ot v(\psi^\mu_{t,y}(g))v(\psi^\mu_{y,t'}(h))$$
$$=\sum_{x\in\gamma\cdot G,y\in\mu\cdot G}v_{sx}(\alpha_t(g))v_{xs'}(\alpha_{y}(h))v_{ty}(g)v_{yt'}(h)u(\psi^\gamma_{s,x}(\alpha_t(g))u(\psi^\gamma_{x,s'}(\alpha_{y}(h)))\ot v(\psi^\mu_{t,y}(g))v(\psi^\mu_{y,t'}(h)).$$
Moreover, since for all $g\in G_r$ and all $s,t$ such that $st=r$, one
has, whenever $t\cdot g=y$ and $s\cdot \alpha_t(g)=x$, that
$xy=(s\cdot\alpha_t(g))(t\cdot g)=(st)\cdot g=r\cdot g=r$, it follows that the only
non-zero terms in the last sum are for $x\in\gamma\cdot G$ and
$y\in\mu\cdot G$ such that $xy=r$. By the properties of the matrix units we
see immediately that $w(gh)=w(g)w(h)$. To end the proof of $(1)$, it
suffices to check that $w(1)=1$, which is clear, and that
$w(g)^*=w(g^{-1})$ for all $g\in G_r$.  So let $g\in G_r$. One has:
$$w(g)^*=\sum_{s,s'\in\gamma\cdot G,\,t,t'\in\mu\cdot G,\,st=r=s't'}e_{ss'}\ot e_{tt'}\ot v_{s's}(\alpha_{t'}(g))v_{t't}(g)u((\psi^\gamma_{s',s}(\alpha_{t'}(g)))^{-1})\ot v((\psi^\mu_{t',t}(g))^{-1}).$$
Note that for all $t,t'\in\Gamma$ and all $g\in G$, one has $v_{s's}(g)=v_{ss'}(g^{-1})$. Also, using the bicrossed product relations one finds that $\alpha_r(g)^{-1}=\alpha_{r\cdot g}(g^{-1})$ for all $r\in\Gamma$ and $g\in G$. In particular, $v_{s's}(\alpha_{t'}(g))v_{t't}(g)=v_{ss'}(\alpha_t(g^{-1}))v_{tt'}(g^{-1})$ and, when $t'\cdot g=t$, one has $\psi^\gamma_{s',s}(\alpha_{t'}(g)))^{-1}=\psi^{\gamma}_{s,s'}(\alpha_{t}(g^{-1}))$. It follows immediately that $w(g)^*=w(g^{-1})$.

$(2)$. Is a direct computation.

$(3)$. One has  $\dim(\morph_{\mathbb{G}}(\gamma_1(u),\gamma_2(v)\ot\gamma_3(w)))=h(\chi(\gamma_1(u))^*\chi(\gamma_2(v))\chi(\gamma_3(w)))$ which is equal to:

\begin{eqnarray*}
&&\sum_{r\in\gamma_1\cdot G,s\in\gamma_2\cdot G,t\in\gamma_3\cdot G}h(\chi(\overline{u})\circ\psi^{\gamma_1}_{r,r}v_{rr}u_r^*u_sv_{ss}\chi(v)\circ\psi^{\gamma_2}_{s,s}u_tv_{tt}\chi(w)\circ\psi^{\gamma_3}_{t,t})\\
&&=\sum_{r,s,t}h(u_{r^{-1}st}\alpha_{t^{-1}s^{-1}r}(\chi(\overline{u})\circ\psi^{\gamma_1}_{r,r}v_{rr})\alpha_{t^{-1}}(v_{ss}\chi(v)\circ\psi^{\gamma_2}_{s,s})v_{tt}\chi(w)\circ\psi^{\gamma_3}_{t,t})\\
&&=\sum_{r\in\gamma_1\cdot G}\sum_{s\in\gamma_2\cdot G,t\in\gamma_3\cdot G,st=r}\int_G\chi(\overline{u})\circ\psi^{\gamma_1}_{r,r}v_{rr}\alpha_{t^{-1}}(v_{ss}\chi(v)\circ\psi^{\gamma_2}_{s,s})v_{tt}\chi(w)\circ\psi^{\gamma_3}_{t,t}d\nu\\
&&=\sum_{r\in\gamma_1\cdot G\cap(\gamma_2\cdot G)(\gamma_3\cdot G)}\frac{1}{\vert r\cdot G\vert}\int_{G_r}\chi(\overline{u})\circ\psi^{\gamma_1}_{r,r}\chi(v\underset{r}{\ot}w)d\nu_r\\
&&=\frac{1}{\vert\gamma_1\cdot G\vert}\sum_{r\in\gamma_1\cdot G\cap(\gamma_2\cdot G)(\gamma_3\cdot G)}\dim(\morph_{G_r}(u\circ\psi^{\gamma_1}_{r,r},v\underset{r}{\ot}w)).\\
\end{eqnarray*}

Note that, whenever $\gamma_1\cdot G\cap((\gamma_2\cdot G)( \gamma_3\cdot G))=\emptyset$, there is no non-zero terms in the sum above.
\end{proof}

\subsection{The induced representation}
In this section, we explain how the induced representation maybe
viewed as a particular twisted tensor product.

For $\gamma\in\Gamma$ and $u\,:\,G_\gamma\rightarrow\mathcal{U}(H)$ is a unitary representation
of $G_\gamma$ we define the induced representation:
$${\rm Ind}_\gamma^G(u):=\varepsilon_{G_{\gamma^{-1}}}\underset{1}{\ot} u\,:\,G\rightarrow\mathcal{U}(l^2(\gamma\cdot G)\ot H);\,\,g\mapsto\sum_{r,s\in\gamma\cdot G}e_{rs}\ot v_{rs}(g)u(\psi^\gamma_{rs}(g)).$$
It follows from Theorem \ref{ThmFusion} that ${\rm Ind}_\gamma^G(u)$ is
indeed a unitary representation of $G$. We collect some elementary and
well known facts about this representation in the following
Proposition. Note that, in property $3$, we use the symbol
${\rm Res}^G_{G_\gamma}(u)$ for $u\in {\rm Rep}(G)$ to denote the restriction
of $u$ to a representation of $G_\gamma$. Hence, property $3$ motivates the
name induced representation for the representation ${\rm Ind}_\gamma^G(u)$.

\begin{proposition}\label{PropInduced}
The following holds.
\begin{enumerate}
\item For all $\gamma\in\Gamma$ and all $u\in{\rm Rep}(G_\gamma)$ one has $\chi({\rm Ind}_\gamma^G(u))(g)=\sum_{r\in\gamma\cdot G}v_{rr}(g)\chi(u)(\psi_{rr}^\gamma(g))$ for all $g\in G$.
\item For all $\gamma\in\Gamma$ and all $u,v\in{\rm Rep}(G_\gamma)$ one has $u\simeq v\implies{\rm Ind}_\gamma^G(u)\simeq{\rm Ind}_\gamma^G(v)$.
\item For all $\gamma\in\Gamma$, $u\in{\rm Rep}(G)$ and $v\in{\rm Rep}(G_\gamma)$ one has ${\rm dim}({\rm Mor}_G(u,{\rm Ind}_\gamma^G(v)))={\rm dim}({\rm Mor}_{G_\gamma}({\rm Res}^G_{G_\gamma}(u),v)$.
\end{enumerate}
\end{proposition}

\begin{proof}
$(1)$. It is obvious, by definition of ${\rm Ind}_\gamma^G(u)$.

$(2)$. If $u\simeq v$ then $\chi(u)=\chi(v)$. Hence,  $\chi({\rm Ind}_\gamma^G(u))=\chi({\rm Ind}_\gamma^G(v))$ by $(1)$. So ${\rm Ind}_\gamma^G(u)\simeq{\rm Ind}_\gamma^G(v)$.




$(3)$. Let $\gamma\in\Gamma$, $u\in{\rm Rep}(G)$ and $v\in{\rm Rep}(G_\gamma)$. One has,
$$
{\rm dim}({\rm Mor}_G(u,{\rm Ind}_\gamma^G(v)))=
\sum_{r\in\gamma\cdot G}\int_G\chi(\overline{u})v_{rr}\chi(v)\circ\psi^\gamma_{rr}d\nu
=\frac{1}{\vert\gamma\cdot G\vert}\sum_{r\in\gamma\cdot G}\int_{G_r}\chi(\overline{u})\chi(v)\circ\psi^\gamma_{rr}d\nu_\gamma.
$$
Since $\psi^\gamma_{rr}\,:\, G_r\rightarrow G_\gamma$ is a Haar probability preserving
homeomorphism we obtain
$${\rm dim}({\rm Mor}_G(u,{\rm Ind}_\gamma^G(v)))=\frac{1}{\vert\gamma\cdot G\vert}\sum_{r\in\gamma\cdot G}\int_{G_\gamma}\chi(\overline{u})\circ(\psi^\gamma_{rr})^{-1}\chi(v)d\nu_\gamma.$$
Finally, since, for all $g\in G$,
$\chi(\overline{u})\circ(\psi^\gamma_{rr})^{-1}(g)=\chi(\overline{u})(g)$ (because
$\chi(\overline{u})$ is a central function on $G$) it follows that:
\begin{displaymath}
  {\rm dim}({\rm Mor}_G(u,{\rm Ind}_\gamma^G(v)))=\frac{1}{\vert\gamma\cdot G\vert}\sum_{r\in\gamma\cdot
    G}\int_{G_\gamma}\chi(\overline{u})\chi(v)d\nu_\gamma={\rm dim}({\rm Mor}_{G_\gamma}({\rm
    Res}^G_{G_\gamma}(u),v). \qedhere
\end{displaymath}
\end{proof}

\section{Length functions}

Recall that given a compact quantum group $\mathbb{H}$, a function
$l:\text{Irr}(\mathbb{H})\rightarrow [0,\infty)$ is called a \textit{length function
  on} ${\rm Irr}(\mathbb{H})$ if $l([\epsilon])=0$,
$l(\overline{x})=l(x)$ and that $l(x)\leq l(y)+l(z)$ whenever
$x\subset y\otimes z$. A length function on a discrete group $\Lambda$ is a function
$l:\Lambda\rightarrow [0,\infty)$ such that $l(1)=0$, $l(r)=l(r^{-1})$ and
$l(rs)\leq l(r) + l(s)$ for all $r,s \in \Lambda$.

Let $(\Gamma, G)$ be a matched pair with bicrossed product
$\mathbb{G}$. In view of the description of the irreducible
representations of $\mathbb{G}$, the fusion rules and the
contragredient representation, it is clear that to get a length
function on ${\rm Irr}(\mathbb{G})$, we need a family of maps
$l_\gamma\,:\,\irr(G_\gamma)\rightarrow[0,+\infty[$, for
$\gamma\in\Gamma$, satisfying the hypothesis of the following definition.

\begin{definition}\label{matchedlength}
  Let $(\Gamma,G)$ be a matched pair,
  $l\,:\,{\rm Irr}(G)\rightarrow[0,+\infty[$ and
  $l_\Gamma\,:\,\Gamma\rightarrow[0,+\infty[$ be length functions. The pair
  $(l,l_\Gamma)$ is \textit{matched} if, for all $\gamma\in\Gamma$, there exists a
  function $l_\gamma\,:\,{\rm Irr}(G_\gamma)\rightarrow[0,+\infty[$ such that
\begin{enumerate}[(i)]
\item $l_{1}=l$ and $l_\gamma(\varepsilon_{G_\gamma})=l_\Gamma(\gamma)$.
\item For any \( \gamma \in \Gamma \), \( r \in \gamma \cdot G \), and
  \( x \in \irr(G_{\gamma}) \), we have
  $l_{\gamma}(x) = l_{r}([u^{x} \circ \psi^{\gamma}_{r,r}])$.
  \item For any \( \gamma \in \Gamma \), \( x \in \irr(G_{\gamma}) \), we have
    $l_{\gamma}(x) = l_{\gamma^{-1}}([\overline{u^{x}} \circ \alpha_{\gamma^{-1}}])$.
  \item For any \( \gamma_{1},\gamma_{2},\gamma_{3} \in \Gamma \),
    \( x \in \irr(G_{\gamma_{1}}) \),
    \( y \in \irr(G_{\gamma_{2}}) \),
    \( z \in \irr(G_{\gamma_{3}}) \), if
    \( \gamma_{3} \in (\gamma_{1}\cdot G)(\gamma_{2} \cdot G) \), and
    \begin{equation}
      \dim\morph_{G_{r}}(u^{z} \circ \psi^{\gamma_{3}}_{r,r}, u^{x} \otimes_{r} u^{y})
      \neq 0
    \end{equation}
    for some \( r \in \gamma_{3} \cdot G \), then
    \begin{equation}    
      l_{\gamma_{3}}(z) \leq l_{\gamma_{1}}(x) + l_{\gamma_{2}}(y).
    \end{equation}
  \end{enumerate}
\end{definition}

The next Proposition shows that our notion of matched pair for length
functions is the good one, as expected.

\begin{proposition}\label{Proplength}
  Let $(\Gamma, G)$ be a matched pair with bicrossed product $\mathbb{G}$.
\begin{enumerate}
\item If $l$ is a length function on $\irr(\mathbb{G})$ then the maps
  $l_G\,:\,\irr(G)=\irr(G_1)\rightarrow[0,+\infty[$,
  $x\mapsto l([1(x)])$ and
  $l_\Gamma\,:\,\Gamma\rightarrow[0,+\infty[$,
  $\gamma\mapsto l([\gamma(\varepsilon_{G_\gamma})])$ are length functions and the pair
  $(l_\Gamma, l_G)$ is matched.
\item If $l_\Gamma$ is any $\beta$-invariant length function on
  $\Gamma$ then the map $l'\,:\,\irr(\mathbb{G})\mapsto [0,+\infty[$,
  $[\gamma(u^x)]\mapsto l_\Gamma(\gamma)$ is a well defined length function on
  $\irr(\mathbb{G})$.
\item If $(l_\Gamma,l_G)$ is a matched pair of length functions on
  $(\Gamma,\irr(G))$ then $l_\Gamma$ is $\beta$-invariant and the maps
  $l,\widetilde{l}\,:\,\irr(\mathbb{G})\rightarrow[0,+\infty[$,
  $l([\gamma(u^x)]):= l_\gamma(x)$ and
  $\widetilde{l}([\gamma(u^x)]):= l_\gamma(x)+l_\Gamma(\gamma)$ are well-defined length
  functions.
\end{enumerate}
\end{proposition}

\begin{proof}
  $(1)$. Since $1(\varepsilon_G)$ is the trivial representation of
  $\mathbb{G}$ one has $l_\Gamma(1)=0$. Let $\gamma,\mu\in\Gamma$ and note that
  $\gamma\mu\in(\gamma\cdot G)(\mu\cdot G)$. Moreover,
\begin{eqnarray*}
  \dim(\morph(\varepsilon_{G_{\gamma\mu}},\varepsilon_{G_{\gamma}}\underset{\gamma\mu}{\ot}\varepsilon_{G_{\mu}}))&=&
                                                                   \int_{G_{\gamma\mu}}\chi(\varepsilon_{G_{\gamma}}\underset{\gamma\mu}{\ot}\varepsilon_{G_{\mu}})d\nu_{G_{\gamma\mu}}
                                                                   =\vert\gamma\mu\cdot G\vert\sum_{s\in\gamma\cdot G,t\in\mu\cdot G,st=\gamma\mu}\int_{G_{\gamma\mu}}(v_{ss}\circ\alpha_t)v_{tt}d\nu\\
                                                               &=&\vert\gamma\mu\cdot G\vert\sum_{s\in\gamma\cdot G,t\in\mu\cdot G,st=\gamma\mu}\nu(\alpha_{t^{-1}}(G_s)\cap G_t\cap G_{\gamma\mu})\\
                                                               &\geq&\nu(\alpha_{\mu^{-1}}(G_\gamma)\cap G_\mu\cap G_{\gamma\mu}).
\end{eqnarray*}
Hence, since $\alpha_{\mu^{-1}}(G_\gamma)\cap G_\mu\cap G_{\gamma\mu}$ is open and non empty (it
contains $1$) we deduce that
$$\dim(\morph(\varepsilon_{G_{\gamma\mu}},\varepsilon_{G_{\gamma}}\underset{\gamma\mu}{\ot}\varepsilon_{G_{\mu}}))>0.$$
So
$\varepsilon_{G_{\gamma\mu}}\subset \varepsilon_{G_{\gamma}}\underset{\gamma\mu}{\ot}\varepsilon_{G_{\mu}}$ and, by
the fusion rules of $\mathbb{G}$ in Theorem \ref{ThmFusion}, $(\gamma\mu)(\varepsilon_{G_{\gamma\mu}})\subset\gamma(\varepsilon_{G_{\gamma}})\ot\mu(\varepsilon_{G_{\mu}})$.
Hence, since
$l$ is a length function,
$l_\Gamma(\gamma\mu)=l([\gamma\mu(\varepsilon_{G_{\gamma\mu}})])\leq l([\gamma(\varepsilon_{G_{\gamma}})])+
l([\mu(\varepsilon_{G_{\mu}})])=l_\Gamma(\gamma)+l_\Gamma(\mu)$. Finally, note that, by point $4$ of Theorem \ref{IrrBC}, for all
$\gamma\in\Gamma$, one has $\gamma^{-1}(\varepsilon_{G_{\gamma^{-1}}})\simeq\overline{\gamma(\varepsilon_G)}$. Hence,
$$l_\Gamma(\gamma^{-1})=l([\gamma^{-1}(\varepsilon_{G_{\gamma^{-1}})]}=l([\overline{\gamma(\varepsilon_G)}])=l([\gamma(\varepsilon_G)])=l_\Gamma(\gamma).$$
So $l_\Gamma$ is a length function on $\Gamma$. It is obvious that
$l_G$ is a length function on $\irr(G)$. Let us prove that the
pair $(l_\Gamma, l_G)$ is matched. Indeed, defining
$l_\gamma\,:\,\irr(G_\gamma)\rightarrow[0,+\infty[$ by
$l_\gamma(x)=l([\gamma(u^x)])$, point $(i)$ of Definition \ref{matchedlength} is clear while point $(ii)$ follows from point $3$ of Theorem \ref{IrrBC}, since it implies
  \( [\gamma(u^{x})] = [r(u^{x} \circ \psi^{r}_{\gamma, \gamma})] \),
  thus
  \begin{displaymath}
    l_{\gamma}(x) = l( [\gamma(u^{x})] )
    = l( [r(u^{x} \circ \psi^{r}_{\gamma, \gamma})] )
    = l_{r}( [u^{x} \circ \psi^{r}_{\gamma, \gamma}]).
  \end{displaymath}
Next, by point $4$ of Theorem \ref{IrrBC}, we have $\overline{[\gamma(u^{x})]} = [\gamma^{-1}(\overline{u^{x}}) \circ \alpha_{\gamma^{-1}}]$ thus,
  \begin{displaymath}
    l_{\gamma}(x) = l(\overline{[\gamma(u^{x})]})
    = l([\gamma^{-1}(\overline{u^{x}}) \circ \alpha^{-1}])
    = l_{\gamma^{-1}}([\overline{u^{x}} \circ \alpha^{-1}]),
  \end{displaymath}
which proves point $(ii)$ of Definition \ref{matchedlength}. Finally, for point $(iv)$, the fusion rules in Theorem \ref{ThmFusion} imply
  \begin{equation}
    \label{eq:bfb19b457bc2d11f}
    \dim\morph( \gamma_{3}(u^{z}),
    \gamma_{1}(u^{x}) \otimes \gamma_{2}(u^{y}))
    = \frac{1}{\vert\gamma \cdot G\vert} \sum_{r \in \gamma_{3} \cdot G}
    \dim\morph_{G_{r}}(u^{z} \circ \psi^{\gamma_{3}}_{r,r}, u^{x} \otimes_{r} u^{y}).
  \end{equation}
If $\dim\morph_{G_{r}}(u^{z} \circ \psi^{\gamma_{3}}_{r,r}, u^{x} \otimes_{r} u^{y})\neq 0$ for some $r \in \gamma_{3} \cdot G$, then $(\ref{eq:bfb19b457bc2d11f})$ is also nonzero, which means, by irreducibility of $\gamma_{3}(u^{z})$ that  \( [\gamma_{3}(u^{z})] \subseteq [\gamma_{1}(u^{x})] \otimes [\gamma_{2}(u^{y})] \). Hence, since $l$ is a length function on
$\irr(\mathbb{G})$,
  \begin{displaymath}
    l_{\gamma_{3}}(z) = l([\gamma_{3}(u^{z})])
    \leq l([\gamma_{1}(u^{x})]) + l([\gamma_{2}(u^{y})])
    = l_{\gamma_{1}}(x) + l_{\gamma_{2}}(y).
  \end{displaymath}

$(2)$. Since $l_\Gamma$ is $\beta$-invariant, the map $l'$ is well defined by
Theorem \ref{IrrBC}. It is clear that
$l'(\varepsilon_\mathbb{G})=0$ and, by point $4$ (and $5$) of Theorem
\ref{IrrBC} and since $l'$ is a length function we also have that
$l'(z)=l'(\overline{z})$ for all $z\in\irr(\mathbb{G})$. Let now
$\gamma_1,\gamma_2,\gamma_3\in\Gamma$, $x\in\irr(G_{\gamma_1})$,
$y\in\irr(G_{\gamma_2})$ and $z\in\irr(G_{\gamma_3})$ be such that
$\gamma_1(u^x)\subset \gamma_2(u^y)\ot\gamma_3(u^z)$ then, by point
$3$ in Theorem \ref{ThmFusion}, there exists $r\in\gamma_1\cdot G$,
$s\in\gamma_2\cdot G$ and $t\in\gamma_3\cdot G$ such that $r=st$ (and
$u^x\circ\psi^{\gamma_1}_{r,r}\subset u^y\underset{r}{\ot}
u^z$). Then,
$$l'([\gamma_1(u^x)])=l_\Gamma(\gamma_1)=l_\Gamma(r)\leq
l_\Gamma(s)+l_\Gamma(t)=l_\Gamma(\gamma_2)+l_\Gamma(\gamma_3)=l'([\gamma_2(u^y)])+l'([\gamma_3(u^z)]).$$

$(3)$. Let $(l_\Gamma,l_G)$ be a matched pair of length functions. By
points $1$ and $2$ of Definition \ref{matchedlength} we have, for all
$\gamma\in\Gamma$ and all $r\in\gamma\cdot G$,
$l_\Gamma(\gamma)=l_\gamma(\varepsilon_{G_\gamma})=l_r([\varepsilon_{G_\gamma}\circ\psi^\gamma_{r,r}])=l_r(\varepsilon_{G_r})=l_\Gamma(r)$. Hence,
$l_\Gamma$ is $\beta$-invariant. By assertion $(2)$ we just proved above, we get a length function
$l'$ on $\irr(\mathbb{G})$. Now, it is clear from Definition
\ref{matchedlength}, the fusion rules and the adjoint representation
of a bicrossed product (point $3$ of Theorem \ref{ThmFusion} and point
$4$ of Theorem \ref{IrrBC}) that
$l\,:\,[\gamma(u^x)]\mapsto l_\gamma(x)$ is a length function on
$\irr(\mathbb{G})$. Since $\widetilde{l}=l+l'$, $\widetilde{l}$ is
also a length function on $\irr(\mathbb{G})$.
\end{proof}

\section{Rapid decay and polynomial growth}

In this section we study property $(RD)$ and polynomial growth for bicrossed-products.

\subsection{Generalities}

We use the notion of property $(RD)$ developed by Vergnioux in \cite{Ve07} (see also \cite{BVZ14}) and
recall the definition below. Since we are only dealing with Kac
algebras, we recall the definition of the Fourier transform and rapid
decay only for Kac algebras.

Let $\mathbb{H}$ be a compact quantum group. We use the notation
$l^\infty(\widehat{\mathbb{H}}):=\bigoplus_{x\in\irr(\mathbb{H})}\mathcal{B}(H_x)$ to denote the
$l^\infty$ direct sum. The $c_0$ direct sum is denoted by
$c_0(\widehat{\mathbb{H}})\subset l^\infty(\widehat{\mathbb{H}})$ and the algebraic direct sum is
denoted by $c_c(\widehat{\mathbb{H}})\subset c_0(\widehat{\mathbb{H}})$. An element
$a\in c_c(\widehat{\mathbb{H}})$ is said to have finite support and its finite
support is denoted by ${\rm Supp}(a):=\{x\in\irr(\mathbb{H})\,:\,ap_x\neq 0\}$,
where $p_x$, for $x\in\irr(\mathbb{H})$ denotes the central minimal projection of
$l^\infty(\widehat{\mathbb{H}})$ corresponding to the block $\mathcal{B}(H_x)$.

For a compact quantum group $\mathbb{H}$ which is always supposed to
be of Kac type, and $a\in C_c(\widehat{\mathbb{H}})$ we define its
Fourier transform as:
\begin{align*}
  \mathcal{F}_{\mathbb{H}}(a)=\sum_{x\in {\rm Irr}(\mathbb{H})} {\rm dim}(x)({\rm Tr}_x\ot\id)(u^x(ap_x\ot 1))\in{\rm Pol}(\mathbb{H}),
\end{align*}
and its ``Sobolev 0-norm'' by
$\|a\|^2_{\mathbb{H},0}=\sum_{x\in {\rm Irr}(\mathbb{H})}{\rm dim}(x){\rm
  Tr}_x((a^\ast a)p_x)$.

Given a length function $l:\text{Irr}(\mathbb{H})\rightarrow [0,\infty)$, consider the element
$L=\sum_{x\in \text{Irr}(\mathbb{H})}l(x)p_x$ which is affilated to
$c_0(\widehat{\mathbb{H}})$. Let $q_n$ denote the spectral projections of $L$
associated to the interval $[n,n+1)$.

The pair $(\widehat{\mathbb{H}},l)$ is said to have:
\begin{itemize}
\item\textit{Polynomial growth} if there exists a polynomial
  $P\in \mathbb{R}[X]$ such that for every $k\in \mathbb{N}$ one has
$$\sum_{x\in\irr(\mathbb{H}),\,k\leq l(x)<k+1}\dim(x)^2\leq P(k)$$
\item\textit{Property} $(RD)$ if there exists a polynomial
  $P\in \mathbb{R}[X]$ such that for every $k\in \mathbb{N}$ and
  $a\in q_kc_c(\widehat{\mathbb{H}})$, we have
  $\|\mathcal{F}(a)\|_{C(\mathbb{H})}\leq P(k)\|a\|_{\mathbb{H},0}$.
\end{itemize}
Finally, \textit{$\widehat{\mathbb{H}}$ is said to have polynomial growth
  (resp.\ property $($RD$)$} if there exists a length function $l$ on
${\rm Irr}(\mathbb{H})$ such that $(\widehat{\mathbb{H}},l)$ has polynomial growth
(resp.\ property $(RD)$).

It is known from \cite{Ve07} that if $(\widehat{\mathbb{H}},l)$ has polynomial
growth then $(\widehat{\mathbb{H}},l)$ has rapid decay and the converse also
holds when we assume $\mathbb{H}$ to be co-amenable. Moreover, it is shown also
shown in \cite{Ve07} that duals of compact connected real Lie groups
have polynomial growth. The fact that polynomial growth implies $(RD)$
can easily be deduced from the following lemma.

\begin{lemma}\label{PGRD} Let $\mathbb{H}$ be a CQG, $F\subset{\rm Irr}(\mathbb{H})$ a finite
  subset and $a\in l^\infty(\widehat{\mathbb{H}})$ with $ap_x=0$ for all
  $x\notin F$. Then,
$$\Vert\mathcal{F}_{\mathbb{H}}(a)\Vert\leq 2\sqrt{\sum_{x\in F}{\rm dim}(x)^2}\Vert a\Vert_{\mathbb{H},0}.$$
\end{lemma}
\begin{proof}
  One can copy the proof of Proposition 4.2, assertion $(a)$, in
  \cite{BVZ14} or the proof of Proposition 4.4, assertion $(ii)$, in
  \cite{Ve07}.
\end{proof}

\subsection{Technicalities}

Let $(\Gamma,G)$ be a matched pair with actions $(\alpha,\beta)$ and denote by
$\mathbb{G}$ the bicrossed product.

Recall that
${\rm Irr}(\mathbb{G})=\sqcup_{\gamma\in I}{\rm Irr}(G_\gamma)$, where
$I\subset\Gamma$ is a complete set of representatives for $\Gamma/G$. For
$\gamma\in I$ and $x\in{\rm Irr}(G_\gamma)$, we denote by $\gamma(x)$ the corresponding
element in ${\rm Irr}(\mathbb{G})$. If a complete set of
representatives of ${\rm Irr}(G_\gamma)$, $x\in{\rm Irr}(G_\gamma)$ is given by
$u^x\in\mathcal{B}(H_x)\ot C(G_\gamma)$ then a representative for
$\gamma(x)$ is given by
$$u^{\gamma(x)}:=\sum_{r,s\in\gamma\cdot G}e_{rs}\ot(1\ot u_r v_{rs})u^x\circ\psi_{r,s}\in\mathcal{B}(l^2(\gamma\cdot
G))\ot C(\mathbb{G}).$$

The lemma below gives a way of obtaining an element
$\widetilde{a} \in c_{c}(\widehat{G})$ from an
$a \in c_{c}(\widehat{G_{\gamma}})$ in a suitable way so that they are
compatible with the Fourier transforms.

\begin{lemma}\label{basicLemma}Let $\gamma\in\Gamma$ and
  $a\in c_c(\widehat{G}_\gamma)$. Define
  $\widetilde{a}\in c_c(\widehat{G})$ by:
$$\widetilde{a}p_y=\sum_{x\in{\rm supp}(a)\text{ and }y\subset{\rm Ind}_\gamma^G(x)}\frac{\dim(x)}{\dim(y)}\sum_{i=1}^{\dim(\morph_G(y,{\rm Ind}_\gamma^G(x)))}(S^y_i)^*(e_{\gamma\gamma}\ot ap_x)S^y_i,$$
where $S^y_i\in{\rm Mor}(y,{\rm Ind}_\gamma^G(x))$ is a basis of isometries
with pairwise orthogonal images. The following holds.
\begin{enumerate}
\item If $(l_\Gamma,l)$ is a matched pair of length functions on
  $(\Gamma,\irr(G))$ then, for all $y\in{\rm supp}(\widetilde{a})$ one
  has
$$l(y)\leq{\rm max}(\{l_\gamma(x)\,:\,x\in{\rm supp}(a)\})+l_\Gamma(\gamma),$$
where $(l_\gamma)_{\gamma\in\Gamma}$ is any family of maps realizing the compatibility
relations of Definition \ref{matchedlength}.
\item $\mathcal{F}_{G_\gamma}(a)=v_{\gamma\gamma}\mathcal{F}_G(\widetilde{a})$.
\item $\Vert \widetilde{a}\Vert_{G,0}\leq\Vert a\Vert_{G_\gamma,0}$.
\end{enumerate}
\end{lemma}

\begin{proof}
  $(1)$. Since any $y\in{\rm supp}(\widetilde{a})$ is such that
  $y\subset{\rm Ind}_\gamma^G(x)=\varepsilon_{G_{\gamma^{-1}}}\underset{1}{\ot} x$ for some
  $x\in{\rm supp}(a)$, it follows that any
  $y\in{\rm supp}(\widetilde{a})$ satisfies
  $l(y)=l_1(y)\leq
  l_{\gamma^{-1}}(\varepsilon_{G_{\gamma^{-1}}})+l_\gamma(x)=l_\Gamma(\gamma^{-1})+l_\gamma(x)=l_\Gamma(\gamma)+l_\gamma(x)$
  for some $x\in{\rm supp}(a)$.

$(2)$. One has:
\begin{eqnarray*}
  v_{\gamma\gamma}\mathcal{F}_G(\widetilde{a})&=&v_{\gamma\gamma}\sum_{y}{\rm dim}(y)({\rm Tr}_y\ot\id)(u^y\widetilde{a}p_y\ot 1)\\
                                    &=&v_{\gamma\gamma}\sum_{x\in{\rm supp}(a),\,y\subset{\rm Ind}_\gamma^G(x)}\sum_{i=1}^{{\rm dim}({\rm Mor}(y,{\rm Ind}_\gamma^G(x)))}{\rm dim}(x)({\rm Tr}_y\ot\id)(u^y((S^y_i)^*(e_{\gamma\gamma}\ot ap_x)S^y_i)\ot 1)\\
                                    &=&v_{\gamma\gamma}\sum_{x,y,i}{\rm dim}(x)({\rm Tr}_y\ot\id)(((S^y_i)^*\ot 1){\rm Ind}_\gamma^G(u^x)(e_{\gamma\gamma}\ot ap_x\ot 1)(S^y_i\ot 1))\\
                                    &=&v_{\gamma\gamma}\sum_{x,y,i}{\rm dim}(x)({\rm Tr}_{l^2(\gamma\cdot G)\ot H_x}\ot\id)({\rm Ind}_\gamma^G(u^x)(e_{\gamma\gamma}\ot ap_x\ot 1)(S^y_i(S^y_i)^*\ot 1))\\
                                    &=&v_{\gamma\gamma}\sum_{x\in{\rm supp}(a)}{\rm dim}(x)({\rm Tr}_{l^2(\gamma\cdot G)\ot H_x}\ot\id)({\rm Ind}_\gamma^G(u^x)(e_{\gamma\gamma}\ot ap_x\ot 1))\\
                                    &=&v_{\gamma\gamma}\sum_{x\in{\rm supp}(a)}{\rm dim}(x)({\rm Tr}_{x}\ot\id)(u^x ap_x\ot 1))=\mathcal{F}_{G_\gamma}(a),\\
\end{eqnarray*}
where, in the 3rd equation we use the fact that $(S_i^y)^*\in{\rm Mor}({\rm Ind}_\gamma^G(x),y)$ and, in the last equation we use the definition of the representation  ${\rm Ind}_\gamma^G(u^x)$.

$(3)$. One has:
\begin{align*}
  \Vert \widetilde{a}\Vert_{G,0}^2&= \sum_y\dim(y){\rm Tr}_y(\widetilde{a}^*\widetilde{a}p_y)\\
                          &= \sum_{x\in{\rm supp}(a),\,y\subset{\rm Ind}_\gamma^G(x)}\sum_{i,j=1}^{{\rm dim}({\rm Mor}(y,{\rm Ind}_\gamma^G(x)))}\dim(y)\frac{{\rm dim}(x)^2}{{\rm dim}(y)^2}{\rm Tr}_y( (S^y_i)^*(e_{\gamma\gamma}\ot a^*p_x)S^y_i (S^y_j)^*(e_{\gamma\gamma}\ot ap_x)S^y_j)\\
                          &= \sum_{x,y,i}\dim(x)\left(\frac{\dim(x)}{\dim(y)}\right){\rm Tr}_y( (S^y_i)^*(e_{\gamma\gamma}\ot a^*p_x)S^y_i (S^y_i)^*(e_{\gamma\gamma}\ot ap_x)S^y_i)\\
\end{align*}
Since, for all $y,i$, $S^y_i(S^y_i)^*$ is a projection, one has
$S^y_i(S^y_i)^*\leq 1$ hence,
$${\rm Tr}_y( (S^y_i)^*(e_{\gamma\gamma}\ot a^*p_x)S^y_i (S^y_i)^*(e_{\gamma\gamma}\ot ap_x)S^y_i)\leq {\rm Tr}_y( (S^y_i)^*(e_{\gamma\gamma}\ot a^*ap_x)S^y_i).$$
Moreover, by Proposition \ref{PropInduced}, one has $y\subset {\rm Ind}_\gamma^G(x)$ if and only if
$$\dim(\morph_{G_\gamma}({\rm Res}^G_{G_\gamma}(y),x))=\dim(\morph_G(y,{\rm
  Ind}_\gamma^G(x)))\neq 0.$$
  Since $x$ is irreducible, we find that
$y\subset {\rm Ind}_\gamma^G(x)\Leftrightarrow x\subset {\rm Res}^G_{G_\gamma}(y)$. In particular, for any
$y\subset {\rm Ind}_\gamma^G(x)$ one has $\dim(x)\leq\dim(y)$. Hence,
\begin{displaymath}
  \begin{split}
    \Vert \widetilde{a}\Vert_{G,0}^2&\leq \sum_{x,y,i}\dim(x){\rm Tr}_y(
    (S^y_i)^*(e_{\gamma\gamma}\ot a^*ap_x)S^y_i)
    =\sum_{x,y,i}\dim(x){\rm Tr}_{l^2(\gamma\cdot G)\ot H_x}( e_{\gamma\gamma}\ot a^*ap_x(S^y_i)^*S^y_i)\\
    &= \sum_{x\in{\rm supp}(a)}\dim(x){\rm Tr}_{l^2(\gamma\cdot G)\ot H_x}(
    e_{\gamma\gamma}\ot a^*ap_x)=\sum_{x\in{\rm supp}(a)}\dim(x){\rm Tr}_x(a^*ap_x)=\Vert
    a\Vert^2_{G_\gamma,0}. \qedhere
  \end{split}
\end{displaymath}
\end{proof}

\begin{lemma}\label{LemmaPGRD}
  Let $(l_\Gamma,l)$ be a matched pair of length functions on
  $(\Gamma,\irr(G))$. If $(\widehat{G},l)$ has polynomial growth then,
  there exists $C>0$ and $N\in\N$ such that:
\begin{itemize}
\item $\Vert\mathcal{F}_G(a)\Vert\leq C(k+1)^N\Vert a\Vert_{G,0}$ for all
  $a\in c_c(\widehat{G})$ with
  ${\rm supp}(a)\subset\{x\in\irr(G)\,:\, l(x)<k+1\}$.
\item
  $\vert\gamma\cdot G\vert{\rm dim}(x)\leq C(l_\Gamma(\gamma)+l_\gamma(x)+1)^N$ for all
  $\gamma\in\Gamma$, $x\in\irr(G_\gamma)$.
\item For all $\gamma\in\Gamma$, $\sum_{x\in\irr(G_\gamma),\,l_\gamma(x)<k+1}\dim(x)^2\leq C^2(k+l_\Gamma(\gamma)+1)^{2N}$.
\end{itemize}
\end{lemma}

\begin{proof}
  Let $P\in\R[X]$ be such that
  $\sum_{x\in\irr(G),\,k\leq l(x)<k+1}{\rm dim}(x)^2\leq P(k)$ for all
  $k\in\N$ and let $C_1>0$ and $N_1\in\N$ be such that
  $P(k)\leq C_1(k+1)^{N_1}$ for all $k\in\N$. By Lemma \ref{PGRD} one has,
  for all $a\in c_c(\widehat{G})$, with
  ${\rm supp}(a)\subset\{x\in\irr(G)\,:\,k\leq l(x)<k+1\}$,
  $\Vert\mathcal{F}_G(a)\Vert\leq 2\sqrt{P(k)}\Vert a\Vert_{G,0}\leq
  \sqrt{C_1}(k+1)^{\frac{N_1}{2}}\Vert a\Vert_{G,0}$. Now, suppose that
  ${\rm supp}(a)\subset\{x\in\irr(G)\,:\,l(x)< k+1\}$ so that
  $a\in q_kc_c(\widehat{G})$, where $q_k=\sum_{j=0}^k p_j$ and
  $p_j=\sum_{x\in\irr(G),\,k\leq l(x)< k+1}$. One has,
\begin{equation}\label{Eq1}
\Vert\mathcal{F}_G(a)\Vert=\sum_{j=0}^k\Vert\mathcal{F}_G(ap_j)\Vert\leq\sum_{j=0}^k \sqrt{C_1}(j+1)^{\frac{N_1}{2}}\Vert a\Vert_{G,0}\leq\sqrt{C_1}(k+1)^{\frac{N_1}{2}+1}\Vert a\Vert_{G,0}.
\end{equation}
Now, let $\gamma\in\Gamma$ and $x\in\irr(G_\gamma)$. By Proposition \ref{PropInduced} one has:
\begin{eqnarray*}
\vert\gamma\cdot G\vert\dim(x)&=&\dim({\rm Ind}_\gamma^G(x))=\sum_{y\in\irr(G)}\dim(\morph_G(y,{\rm Ind}_\gamma^G(x)))\dim(y)\\
&=&\sum_{y\in\irr(G),\,y\subset{\rm Ind}_\gamma^G(x)}\dim(\morph_{G_\gamma}({\rm Res}^G_{G_\gamma}(y),x))\dim(y).
\end{eqnarray*}
Note that
$\dim(\morph_{G_\gamma}({\rm Res}^G_{G_\gamma}(y),x))\leq\dim(y)$ for all
$x,y$. Moreover, since
${\rm Ind}_\gamma^G(x)\simeq \varepsilon_{G_{\gamma^{-1}}}\underset{1}{\ot} x$ and the pair
$(l_\Gamma, l)$ is matched, one has
$\{y\in\irr(G),\,y\subset{\rm Ind}_\gamma^G(x)\}\subset\{y\in\irr(G)\,:\,l(y)\leq
l_\Gamma(\gamma)+l_\gamma(x)\}$. Hence,
\begin{eqnarray}
\vert\gamma\cdot G\vert\dim(x)&\leq& \sum_{y\in\irr(G),\,l(y)< l_\Gamma(\gamma)+l_\gamma(x)+1}\dim(y)^2=\sum_{j=0}^{ l_\Gamma(\gamma)+l_\gamma(x)} \sum_{y\in\irr(G),\,j\leq l(y)<j+1}\dim(y)^2\nonumber\\
&\leq&\sum_{j=0}^{ l_\Gamma(\gamma)+l_\gamma(x)} P(j)\leq C_1\sum_{j=0}^{ l_\Gamma(\gamma)+l_\gamma(x)} (j+1)^{N_1}\leq C_1(l_\Gamma(\gamma)+l_\gamma(x)+1)^{N_1+1}.\label{Eq2}
\end{eqnarray}
It follows from Equations $(\ref{Eq1})$ and $(\ref{Eq2})$ that
$C:={\rm Max}(C_1,\sqrt{C_1})$ and $N:=N_1+1$ do the job.

Let us show the last point. Fix $\gamma\in\Gamma$ and let
$F\subset\irr(G_\gamma)$ a finite subset. Define
$p_F\in c_c(\widehat{G}_\gamma)$ by $p_F=\sum_{x\in F}p_x$ and note that
$\mathcal{F}_{G_\gamma}(p_F)=\sum_{x\in F}\dim(x)\chi(x)$ and
$\Vert a\Vert_{G_\gamma,0}^2=\sum_{x\in F}\dim(x)^2$. Suppose that
$F\subset\{x\in\irr(G_\gamma)\,:\,l_\gamma(x)<k+1\}$. Using Lemma \ref{basicLemma} and
the first part of the proof we find, since
$\widetilde{p_F}\in c_c(\widehat{G})$ with
${\rm supp}(\widetilde{p_F})\subset\{x\in\irr(G)\,:\,l(x)<l_\Gamma(\gamma)+k+1\}$,
\begin{eqnarray*}
\left\Vert\sum_{x\in F}\dim(x)\chi(x)\right\Vert^2&=&\Vert\mathcal{F}_{G_\gamma}(p_F)\Vert^2=\Vert v_{\gamma\gamma}\mathcal{F}_G(\widetilde{p_F})\Vert^2\leq \Vert\mathcal{F}_G(\widetilde{p_F})\Vert^2\leq C^2(k+l_\Gamma(\gamma)+1)^{2N}\Vert \widetilde{p_F}\Vert^2_{G,0}\\
&\leq& C^2(k+l_\Gamma(\gamma)+1)^{2N}\Vert p_F\Vert^2_{G_\gamma,0}=C^2(k+l_\Gamma(\gamma)+1)^{2N}\sum_{x\in F}\dim(x)^2.
\end{eqnarray*}
It follows that:
$$\left(\sum_{x\in F}\dim(x)^2\right)^2=\left(\sum_{x\in F}\dim(x)\chi(x)(1)\right)^2\leq \left\Vert\sum_{x\in F}\dim(x)\chi(x)\right\Vert_{C(G)}^2\leq C^2(k+l_\Gamma(\gamma)+1)^{2N}\sum_{x\in F}\dim(x)^2.$$
Hence, for all non empty finite subsets
$F\subset\{x\in\irr(G_\gamma)\,:\,l_\gamma(x)<k+1\}$ one has
$\sum_{x\in F}\dim(x)^2\leq C^2(k+l_\Gamma(\gamma)+1)^{2N}$. The last assertion
follows.
\end{proof}

\subsection{Polynomial growth for bicrossed product}
 
We start with the following result.

\begin{theorem} Suppose that that $(l_G,l_\Gamma)$ is a matched pair of
  length functions on $(\Gamma,G)$. If both $(\Gamma,l_\Gamma)$ and
  $(\widehat{G},l_G)$ has polynomial growth then
  $(\widehat{\mathbb{G}},\widetilde{l})$ have polynomial growth.
\end{theorem}

\begin{proof}
  Let $I\subset\Gamma$ be a complete set of representatives for
  $\Gamma/ G$ so that
  $\irr(\mathbb{G})=\sqcup_{\gamma\in I}\irr(G_\gamma)$. Let $k\geq 1$ and define
  $$F_k:=\{z\in\irr(\mathbb{G})\,:\,\widetilde{l}(z)<k\}\subset\sqcup_{\gamma\in I_k}T_{\gamma,k},$$
  where $I_k:=\{\gamma\in \Gamma\,:\,l_\Gamma(\gamma)< k+1\}\cap I$ and
  $T_{\gamma,k}:=\{x\in\irr(G_\gamma)\,:\,l_\gamma(x)< k+1\}$. Since
  $(\Gamma,l_\Gamma)$ has polynomial growth, there exists a polynomial
  $P_1$ such that, for all $k\in\N$, $\vert I_k\vert\leq P_1(k)$. Moreover, since
  $(\widehat{G},l_G)$ has polynomial growth, we can apply Lemma
  \ref{LemmaPGRD} to get $C>0$ and $N\in N$ such that, for all
  $k\in \N$ and all $\gamma\in I_k$, one has
  $\sum_{x\in T_{\gamma,k}}\dim(x)^2\leq C^2(2k+2)^{2N}$ and,
  $\vert\gamma\cdot G\vert=\vert\gamma\cdot G\vert\dim(\varepsilon_G)\leq C(2k+3)^N$. Hence, for all $k\geq 1$,
  \begin{align*}
    \sum_{z\in F_k}\dim(z)^2&=\sum_{\gamma\in I_k}\vert\gamma\cdot G\vert^2\sum_{x\in T_{\gamma, k}}\dim(x)^2\leq C^2(2k+2)^{2N}\sum_{\gamma\in I_k}\vert\gamma\cdot G\vert^2\leq C^4(2k+2)^{2N}(2k+3)^{2N}\vert I_k\vert\\
                       &\leq C^4(2k+2)^{2N}(2k+3)^{2N}P_1(k). \qedhere
  \end{align*}
\end{proof}

To complete the proof of Theorem B, we need the following Proposition.

\begin{proposition} Assume that there exists a length function $l$ on
  $\irr(\mathbb{G})$ such that $(\widehat{\mathbb{G}},l)$ has
  polynomial growth and consider the matched pair of length functions
  $(l_\Gamma, l_G)$ associated to $l$ given in Proposition
  \ref{Proplength}. Then $(\Gamma,l_\Gamma)$ and $(\widehat{G},l_G)$ both have
  polynomial growth.
\end{proposition}

\begin{proof}
  Assume that $(\widehat{\mathbb{G}},l)$ has polynomial growth. Since
  the map $\irr(G)\rightarrow\irr(\mathbb{G})$, $x\mapsto 1(x)$ is injective,
  dimension preserving and length preserving (by definition of $l_G$),
  it is clear that $(\widehat{G},l_G)$ has polynomial growth. Let us
  show that $(\Gamma,l_\Gamma)$ also has polynomial growth. Let $P$ be a
  polynomial witnessing $(RD)$ for $(\widehat{\mathbb{G}},l)$ and
  $k\in\N$. Define
  $F_k:=\{\gamma\in\Gamma\,:\,k\leq l_\Gamma(\gamma)<k+1\}$. One has, for all $k\in\N$,
\begin{align*}
  \vert F_k\vert&=\sum_{k\leq l([\gamma(\varepsilon_G)])< k+1}1\leq \sum_{k\leq l([\gamma(\varepsilon_G)])< k+1}\vert\gamma\cdot G\vert^2=\sum_{k\leq l([\gamma(\varepsilon_G)])< k+1}\dim([\gamma(\varepsilon_G)])^2\\
        &\leq\sum_{z\in\irr(\mathbb{G}),\,k\leq l(z)< k+1}\dim(z)^2\leq P(k). \qedhere
\end{align*}
\end{proof}

\subsection{Rapid decay for bicrossed product}

Recall that
$l^\infty(\widehat{\mathbb{G}})=\bigoplus_{\gamma\cdot G\in\Gamma/G}\bigoplus_{x\in{\rm
    Irr}(G_\gamma)}\mathcal{B}(l^2(\gamma\cdot G)\ot H_x)$. Let us denote by
$p_{\gamma(x)}$ the central projection of
$l^\infty(\widehat{\mathbb{G}})$ corresponding to the block
$\mathcal{B}(l^2(\gamma\cdot G)\ot H_x)$ and define, for
$\gamma\cdot G\in\Gamma/G$, the central projection :
$$p_\gamma:=\sum_{x\in{\rm Irr}(G_\gamma)} p_{\gamma(x)}\in l^\infty(\widehat{\mathbb{G}}).$$
Note that
$p_\gamma l^\infty(\widehat{\mathbb{G}})=\bigoplus_{x\in{\rm Irr}(G_\gamma)}\mathcal{B}(l^2(\gamma\cdot
G)\ot H_x)\simeq \mathcal{B}(l^2(\gamma\cdot G))\ot L(G_\gamma)$, where
$L(G_\gamma)=\bigoplus_{x\in{\rm Irr}(G_\gamma)}\mathcal{B}(H_x)$ is the group von-Neumann
algebra of $G_\gamma$ (which is also the multiplier C*-algebra of
$C^*_r(G_\gamma)=\bigoplus^{c_0}_{x\in{\rm Irr}(G_\gamma)}\mathcal{B}(H_x)$). Using this
identification, we define
$\pi_\gamma\,:\,c_0(\widehat{\mathbb{G}})\rightarrow\mathcal{B}(l^2(\gamma\cdot G))\ot C^*_r(G_\gamma)\subset
c_0(\widehat{\mathbb{G}})$ to be the $*$-homomorphism given by
$\pi_\gamma(a)=ap_\gamma$, for all $a\in c_0(\widehat{\mathbb{G}})$. We also write,
for $a\in c_0(\widehat{\mathbb{G}})$,
$\pi_\gamma(a)=\sum_{r,s\in\gamma\cdot G}e_{rs}\ot\pi^\gamma_{r,s}(a)$, where we recall that
$(e_{rs})$ are the matrix units associated to the canonical
orthonormal basis $(e_r)_{r\in\gamma\cdot G}$ of $l^2(\gamma\cdot G)$ and
$\pi^\gamma_{r,s}\,:\,c_0(\widehat{\mathbb{G}})\rightarrow C^*_r(G_\gamma)$ is the
completely bounded map defined by
$\pi^\gamma_{r,s}:=(\omega_{e_s,e_r}\ot\id)\circ\pi_\gamma$ and
$\omega_{e_s,e_r}\in\mathcal{B}(l^2(\gamma\cdot G))$, $\omega_{e_s,e_r}(T)=\langle Te_s,e_r\rangle$.

We start with the following result.

\begin{theorem}\label{ThmRD} Let $(l_\Gamma,l_G)$ be a matched pair of
  length functions on $(\Gamma,\irr(G))$. Suppose that
  $(\widehat{G},l_G)$ has polynomial growth and $(\Gamma,l_\Gamma)$ has
  $(RD)$. Then $(\widehat{\mathbb{G}},\widetilde{l})$ has
  $(RD)$.
\end{theorem}

\begin{proof}
  Let $a\in c_c(\widehat{\mathbb{G}})$ and write
  $a=\sum_{\gamma\in S}\sum_{x\in T_\gamma}ap_{\gamma(x)}$, where
  $S\subset I$ and $T_\gamma\subset {\rm Irr}(G_\gamma)$ are finite subsets.

\textbf{Claim.}\hspace{2ex}\textit{The following holds.
\begin{enumerate}
\item $\mathcal{F}_{\mathbb{G}}(a)=\sum_{\gamma\in S}\vert\gamma\cdot G\vert\left(\sum_{r,s\in\gamma\cdot G}u_rv_{rs}\mathcal{F}_{G_\gamma}(\pi^\gamma_{s,r}(a))\circ\psi^\gamma_{r,s}\right)$.
\item
  $\Vert a\Vert_{\mathbb{G},0}^2=\sum_{\gamma\in S}\vert\gamma\cdot G\vert\left(\sum_{r,s\in\gamma\cdot
      G}\Vert\pi^\gamma_{r,s}(a)\Vert_{G_\gamma,0}^2\right)$.
\end{enumerate}}

\textit{Proof of the Claim.}$(1)$. A direct computation gives:
\begin{eqnarray*}
  \mathcal{F}_{\mathbb{G}}(a)&=&\sum_{\gamma\in S,\,x\in T_\gamma}\vert\gamma\cdot G\vert\dim(x)({\rm Tr}_{l^2(\gamma\cdot G)\ot H_x}(\gamma(u^{x})ap_{\gamma(x)}\ot 1)\\
                             &=&\sum_{\gamma\in S,x\in T_\gamma}\vert\gamma\cdot G\vert\dim(x)\sum_{r,s\in\gamma\cdot G}u_rv_{rs}({\rm Tr}_x\ot \id)(u^x\circ\psi^\gamma_{r,s}\pi^\gamma_{s,r}(a)p_x\ot 1)\\
                             &=&\sum_{\gamma\in S}\vert\gamma\cdot G\vert\sum_{r,s\in\gamma\cdot G}u_rv_{rs}\mathcal{F}_{G_\gamma}(\pi^\gamma_{s,r}(a))\circ\psi^\gamma_{r,s}.
\end{eqnarray*}
$(2)$. Since $\pi_\gamma$ is a $*$-homomorphism, we have
$\pi^\gamma_{r,s}(a^*a)=\sum_{t\in\gamma\cdot G}\pi^\gamma_{t,r}(a)^*\pi^\gamma_{t,s}(a)$ hence,
\begin{eqnarray*}
\Vert a\Vert_{\mathbb{G},0}^2&=&\sum_{\gamma\in S,x\in T_\gamma}\vert\gamma\cdot G\vert\dim(x)\sum_{r,s\in\gamma\cdot G}({\rm Tr}_x\ot \id)(\pi^\gamma_{s,r}(a)^*\pi^\gamma_{r,s}(a))\\
&=&\sum_{\gamma\in S}\vert\gamma\cdot G\vert\sum_{r,s\in\gamma\cdot G}\Vert\pi^\gamma_{r,s}(a)\Vert_{G_\gamma,0}^2.
\end{eqnarray*}

Let us now prove the theorem. Let
$b=\sum_{\gamma\in S'}\sum_{t,t'\in\gamma\cdot G}u_tv_{tt'}F_\gamma\circ\psi^\gamma_{t,t'}\in C(\mathbb{G})$,
where $F_\gamma\in C(G_\gamma)$ and $S'\subset I$ is a finite subset. For all
$r\in\Gamma$, we denote by $\gamma_r$ the unique element in $I$ such that
$\gamma_r\cdot G=r\cdot G$. We may re-order the sums and write:
$$\mathcal{F}_{\mathbb{G}}(a)=\sum_{r\in\Gamma}1_{S\cdot G}(r)\vert r\cdot G\vert\left(\sum_{s\in r\cdot G}u_rv_{rs}\mathcal{F}_{G_{\gamma_r}}(\pi^{\gamma_r}_{s,r}(a))\circ\psi^{\gamma_r}_{r,s}\right)\text{ and }b=\sum_{t\in \Gamma}u_t1_{S'\cdot G}(t)\left(\sum_{t'\in t\cdot G}v_{tt'} F_{\gamma_t}\circ\psi^{\gamma_t}_{t,t'}\right).$$
Also,
$\Vert a\Vert_{\mathbb{G},0}^2=\sum_{r\in\Gamma}1_{S\cdot G}(r)\vert r\cdot G\vert\left(\sum_{s\in r\cdot
    G}\Vert\pi^{\gamma_r}_{r,s}(a)\Vert_{G_{\gamma_r},0}^2\right)$. Then,
$\Vert \mathcal{F}_{\mathbb{G}}(a)b\Vert_{2,h_{\mathbb{G}}}^2$ is equal to :

\begin{eqnarray*}
& &
\left\Vert \sum_{r,t\in\Gamma} u_{rt}1_{S\cdot G}(r)1_{S'\cdot G}(t)\vert r\cdot G\vert\left(\sum_{s\in r\cdot G,t'\in t\cdot G}v_{rs}\circ\alpha_t\mathcal{F}_{G_{\gamma_r}}(\pi^{\gamma_r}_{s,r}(a))\circ\psi^{\gamma_r}_{r,s}\circ\alpha_tv_{tt'}F_{\gamma_t}\circ\psi^{\gamma_t}_{t,t'}\right)\right\Vert^2_{2,h_{\mathbb{G}}}\\
&=&\sum_{x\in\Gamma}\left\Vert\sum_{\substack{r,t\in\Gamma\\rt=x}}1_{S\cdot G}(r)1_{S'\cdot G}(t)\vert r\cdot G\vert \left(\sum_{s\in r\cdot G,t'\in t\cdot G}v_{rs}\circ\alpha_t\mathcal{F}_{G_{\gamma_r}}(\pi^{\gamma_r}_{s,r}(a))\circ\psi^{\gamma_r}_{r,s}\circ\alpha_tv_{tt'}F_{\gamma_t}\circ\psi^{\gamma_t}_{t,t'}\right)\right\Vert^2_{2}\\
&=&\sum_{x\in\Gamma}\left\Vert \sum_{\substack{r,t\in\Gamma\\rt=x}}1_{S\cdot G}(r)1_{S'\cdot G}(t)\vert r\cdot G\vert \left(\sum_{s\in r\cdot G}v_{rs}\circ\alpha_t\mathcal{F}_{G_{\gamma_r}}(\pi^{\gamma_r}_{s,r}(a))\circ\psi^{\gamma_r}_{r,s}\circ\alpha_t\right)\left(\sum_{t'\in t\cdot G}v_{tt'}F_{\gamma_t}\circ\psi^{\gamma_t}_{t,t'}\right)\right\Vert^2_{2}\\
&\leq &\sum_x\left(\sum_{\substack{r,t\in\Gamma\\rt=x}}1_{S\cdot G}(r)1_{S'\cdot G}(t)\vert r\cdot G\vert\left \Vert \sum_{s\in r\cdot G}v_{rs}\circ\alpha_t\mathcal{F}_{G_{\gamma_r}}(\pi^{\gamma_r}_{s,r}(a))\circ\psi^{\gamma_r}_{r,s}\circ\alpha_t\right\Vert_{\infty}\left\Vert \sum_{t'\in t\cdot G}v_{tt'}F_{\gamma_t}\circ\psi^{\gamma_t}_{t,t'}\right\Vert_{2}\right)^2\\
&=&\sum_x\left(\sum_{\substack{r,t\in\Gamma\\rt=x}}\left(1_{S\cdot G}(r)\vert r\cdot G\vert\left \Vert \sum_{s\in r\cdot G}v_{rs}\mathcal{F}_{G_{\gamma_r}}(\pi^{\gamma_r}_{s,r}(a))\circ\psi^{\gamma_r}_{r,s}\right\Vert_{\infty}\right)\left(1_{S'\cdot G}(t)\left\Vert \sum_{t'\in t\cdot G}v_{tt'}F_{\gamma_t}\circ\psi^{\gamma_t}_{t,t'}\right\Vert_{2}\right)\right)^2\\
&=&\Vert\psi*\phi\Vert_{l^2(\Gamma)}^2,
\end{eqnarray*}
where $\Vert\cdot\Vert_2$ and $\Vert\cdot\Vert_{\infty}$ denote respectively the
${\rm L}^2$-norm and the supremum norm on $C(G)$ and
$\psi,\phi\,:\,\Gamma\rightarrow\mathbb{R}_+$ are finitely supported functions defined by :
$$\psi(r):=1_{S\cdot G}(r)\vert r\cdot G\vert\left \Vert \sum_{s\in r\cdot G}v_{rs}\mathcal{F}_{G_{\gamma_r}}(\pi^{\gamma_r}_{s,r}(a))\circ\psi^{\gamma_r}_{r,s}\right\Vert_{\infty}\text{ and }\phi(t):=1_{S'\cdot G}(t)\left\Vert \sum_{t'\in t\cdot G}v_{tt'}F_{\gamma_t}\circ\psi^{\gamma_t}_{t,t'}\right\Vert_{2},$$
Note that
$\Vert\phi\Vert_{l^2(\Gamma)}^2=\Vert b\Vert_{2,h_{\mathbb{G}}}^2$. Moreover, one has, since
$\psi^{\gamma}_{r,s}\,:\, G_{r,s}\rightarrow G_\gamma$ is an homeomorphism,
\begin{eqnarray*}
\Vert\psi\Vert_{l^2(\Gamma)}^2&=&\sum_{r\in\Gamma}1_{S\cdot G}(r)\vert r\cdot G\vert ^2\left \Vert \sum_{s\in r\cdot G}v_{rs}\mathcal{F}_{G_{\gamma_r}}(\pi^{\gamma_r}_{s,r}(a))\circ\psi^{\gamma_r}_{r,s}\right\Vert_{\infty}^2\\
&\leq &\sum_{r\in\Gamma}1_{S\cdot G}(r)\vert r\cdot G\vert^3\sum_{s\in r\cdot G}\left \Vert v_{rs}\mathcal{F}_{G_{\gamma_r}}(\pi^{\gamma_r}_{s,r}(a))\circ\psi^{\gamma_r}_{r,s}\right\Vert_{\infty}^2\\
&=&\sum_{r\in\Gamma}1_{S\cdot G}(r)\vert r\cdot G\vert^3\sum_{s\in r\cdot G}\left \Vert \mathcal{F}_{G_{\gamma_r}}(\pi^{\gamma_r}_{s,r}(a))\right\Vert_{C(G_{\gamma_r})}^2.\\
\end{eqnarray*}
For $k\in\N$ let
$p_k=\sum_{\gamma\in I,x\in{\rm Irr}(G_\gamma)\,:\,k\leq l(\gamma(x))< k+1}p_{\gamma(x)}\in
l^\infty(\widehat{\mathbb{G}})$,
$p^{G_\gamma}_k=\sum_{x\in{\rm Irr}(G_\gamma)\,:\,k\leq l_{G_\gamma}(x)< k+1}p_{x}\in
l^\infty(\widehat{G}_\gamma)$ and suppose from now on that
$a\in p_kc_c(\widehat{\mathbb{G}})$. Hence, we must have
$S\subset\{\gamma\in\Gamma\,:\,l_\Gamma(\gamma)< k+1\}$ and, for all
$\gamma\in S$,
$T_\gamma\subset\{x\in {\rm Irr}(G_\gamma)\,:\,l_{G_\gamma}(x)<k+1\}$. Hence, for all
$\gamma\in S$ and all $r,s\in\gamma\cdot G$ one has
$\pi^\gamma_{r,s}(a)\in q^\gamma_kc_c(\widehat{G}_\gamma)$, where
$q_k^\gamma=\sum_{j=0}^kp_j^{G_\gamma}$.

Since $(\widehat{G},l_{G})$ has polynomial growth, there exists $C>0$
and $N\in\N$ satisfying the properties of Lemma \ref{LemmaPGRD}. In
particular, one has, for all $\gamma\in\Gamma$,
$\vert \gamma\cdot G\vert\leq C(2l_\Gamma(\gamma)+1)^{N}$. Moreover, since
$S\subset\{g\in\Gamma\,:\, l_\Gamma(g)<k+1\}$ and $l_\Gamma$ is
$\beta$-invariant, it follows that
$S\cdot G\subset \{g\in\Gamma\,:\, l_\Gamma(g)<k+1\}$. By Lemma \ref{basicLemma} (and Lemma
\ref{LemmaPGRD}) we deduce that:
\begin{eqnarray*}
\Vert\psi\Vert_{l^2(\Gamma)}^2
&\leq&\sum_{r\in\Gamma}1_{S\cdot G}(r)\vert r\cdot G\vert^3\sum_{s\in r\cdot G}\left \Vert v_{\gamma_r\gamma_r}\mathcal{F}_{G}(\widetilde{\pi^{\gamma_r}_{s,r}(a)})\right\Vert^2\leq
\sum_{r\in\Gamma}1_{S\cdot G}(r)\vert r\cdot G\vert^3\sum_{s\in r\cdot G}\left \Vert\mathcal{F}_{G}(\widetilde{\pi^{\gamma_r}_{s,r}(a)})\right\Vert^2\\
&\leq &\sum_{r\in\Gamma}1_{S\cdot G}(r)\vert r\cdot G\vert^3\sum_{s\in r\cdot G}C^2(k+l_\Gamma(\gamma_r)+1)^{2N}\left \Vert\widetilde{\pi^{\gamma_r}_{s,r}(a)}\right\Vert_{G,0}^2\\
&\leq&C^2(2k+2)^{2N}\sum_{r\in\Gamma}1_{S\cdot G}(r)\vert r\cdot G\vert^3\sum_{s\in r\cdot G}\left \Vert \pi^{\gamma_r}_{s,r}(a)\right\Vert_{G_{\gamma_r},0}^2\\
&\leq&C^4(2k+3)^{4N}\sum_{r\in\Gamma}1_{S\cdot G}(r)\vert r\cdot G\vert\sum_{s\in r\cdot G}\left \Vert \pi^{\gamma_r}_{s,r}(a)\right\Vert_{G_{\gamma_r},0}^2=C^4(2k+3)^{4N}\Vert a\Vert^2_{\mathbb{G},0}.
\end{eqnarray*}

Since $(\Gamma,l_\Gamma)$ has $(RD)$, let $C_2>0$ and $N_2\in\N$ such that for all
$k\in\N$, for all function $\xi$ on $\Gamma$ supported on
$\{g\in\Gamma\,:\,l_\Gamma(g)<k+1\}$, we have
$\Vert\xi*\eta\Vert_{l^2(\Gamma})\leq C_2(k+1)^{N_2}\Vert\xi\Vert_{l^2(\Gamma)}\Vert\eta\Vert_{l^2(\Gamma)}$. Note that
$\psi$ is supported on $S\cdot G$ and
$S\cdot G\subset \{g\in\Gamma\,:\, l_\Gamma(g)<k+1\}$. Hence, it follows from the preceding
computations that:
\begin{eqnarray*}
\Vert \mathcal{F}_{\mathbb{G}}(a)b\Vert_{2,h_{\mathbb{G}}}^2
&\leq&\Vert\psi*\phi\Vert_{l^2(\Gamma)}^2\leq C_2^2(k+1)^{2N_2}\Vert\psi\Vert_{l^2(\Gamma)}\Vert\phi\Vert_{l^2(\Gamma)}
\leq C^4(2k+3)^{4N}C_2^2(k+1)^{2N_2}\Vert a\Vert_{\mathbb{G},0}^2\Vert b\Vert_{2,h_{\mathbb{G}}}^2\\
&=&(P(k)\Vert a\Vert_{\mathbb{G},0}^2\Vert b\Vert_{2,h_{\mathbb{G}}})^2.
\end{eqnarray*}
where $P(X)=C^2C_2^2(2X+3)^{2N}(X+1)^{N_2}$. It concludes the proof.
\end{proof}

To complete the proof of Theorem A, we need the following Proposition.

\begin{proposition} Assume that there exists a length function $l$ on
  $\irr(\mathbb{G})$ such that $(\widehat{\mathbb{G}},l)$ has $(RD)$
  and consider the matched pair of length functions $(l_\Gamma, l_G)$
  associated to $l$ given in Proposition \ref{Proplength}. Then
  $(\Gamma,l_\Gamma)$ has $(RD)$ and $(\widehat{G},l_G)$ has polynomial growth.
\end{proposition}

\begin{proof}
  Suppose that $(\widehat{\mathbb{G}},l)$ has $(RD)$. The fact that
  $(\widehat{G},l_G)$ has $(RD)$ follows from the general theory
  (since $C(G)\subset C(\mathbb{G})$ intertwines the comultiplication and
  the associated injection $\irr(G)\rightarrow \irr(\mathbb{G})$, actually given
  by $(x\mapsto 1(x))$, preserves the length functions). Let us show that
  $(\Gamma,l_\Gamma)$ has $(RD)$. Let $k\in \N$ and
  $\xi\,:\,\Gamma\rightarrow\C$ be a finitely supported function with support in
  $\{\gamma\in\Gamma\,:\, k\leq l_\Gamma(\gamma)<k+1\}$. Define
  $\widetilde{\xi}\in c_c(\widehat{\mathbb{G}})$ by
  $\widetilde{\xi}=\sum_{\gamma\in I}\frac{1}{\vert\gamma\cdot G\vert}\left(\sum_{r\in\gamma\cdot
      G}\xi(r)e_{rr}\right)p_{\gamma(1)}$, where we recall
  $e_{rs}\in \mathcal{B}(l^2(\gamma\cdot G))$ for
  $r,s\in\gamma\cdot G$ are the matrix units associated to the canonical
  orthonormal basis. Then,
  $$\mathcal{F}_{\mathbb{G}}(\widetilde{\xi})=\sum_{\gamma\in I}\sum_{r\in\gamma\cdot G}\xi(r)({\rm Tr}_{l^2(\gamma\cdot G)}\ot\id)(u^{\gamma(1)}(e_{rr}\ot 1))=\sum_{\gamma\in I}\sum_{r\in\gamma\cdot G}\xi(r)u_rv_{rr}\quad\text{ also,}$$
  $$\Vert \widetilde{\xi}\Vert_{\mathbb{G},0}^2=\sum_{\gamma\in I}\vert\gamma\cdot G\vert{\rm Tr}_{l^2(\gamma\cdot G)}(\sum_{r\in\gamma\cdot G}\frac{\vert\xi(r)\vert^2}{\vert\gamma\cdot G\vert^2}e_{rr})=\sum_{\gamma\in I}\frac{1}{\vert\gamma\cdot G\vert}\sum_{r\in\gamma\cdot G}\vert\xi(r)\vert^2\leq\sum_{\gamma\in I}\sum_{r\in\gamma\cdot G}\vert\xi(r)\vert^2=\Vert\xi\Vert_2^2.$$
  Since $\xi$ is supported in
  $\{\gamma\in\Gamma\,:\, k\leq l_\Gamma(\gamma)<k+1\}$ and $l_\Gamma$ is
  $\beta$-invariant, it follows that
  ${\rm supp}(\widetilde{\xi})\subset\{z\in\irr(\mathbb{G})\,:\, k\leq
  l(z)<k+1\}$. Hence, denoting by $P$ a polynomial witnessing $(RD)$
  for $(\widehat{\mathbb{G}},l)$, we have:
  $$\left\Vert \sum_{\gamma\in I}\sum_{r\in\gamma\cdot G}\xi(r)u_rv_{rr}\right\Vert\leq P(k)\Vert\xi\Vert_2.$$
  Denote by $\Psi$ the unital $\ast$-morphism
  $\Psi\,:\, C(\mathbb{G})=\Gamma\ltimes C(G)\rightarrow C^*_r(\Gamma)$ such that
  $\Psi(u_\gamma F)=\lambda_\gamma F(1)$ for all $\gamma\in\Gamma$ and
  $F\in C(G)$. Since $\Psi$ has norm one, denoting by
  $\lambda(\xi)\in C^*_r(\Gamma)$ the convolution operator by $\xi$, we have
  $$\Vert\lambda(\xi)\Vert=\left\Vert\sum_{\gamma\in I}\sum_{r\in\gamma\cdot G}\xi(r)\lambda_r\right\Vert=\left\Vert\Psi\left(\sum_{\gamma\in I}\sum_{r\in\gamma\cdot G}\xi(r)u_rv_{rr}\right)\right\Vert\leq\left\Vert \sum_{\gamma\in I}\sum_{r\in\gamma\cdot G}\xi(r)u_rv_{rr}\right\Vert\leq P(k)\Vert\xi\Vert_2.$$
  This concludes the proof.
\end{proof}

\noindent
{\sc Pierre FIMA} \\ \nopagebreak
  {Univ Paris Diderot, Sorbonne Paris Cit\'e, IMJ-PRG, UMR 7586, F-75013, Paris, France \\
  Sorbonne Universit\'es, UPMC Paris 06, UMR 7586, IMJ-PRG, F-75005, Paris, France \\
  CNRS, UMR 7586, IMJ-PRG, F-75005, Paris, France \\
\em E-mail address: \tt pierre.fima@imj-prg.fr}

\vspace{0.2cm}

\noindent
{\sc Hua WANG} \\ \nopagebreak
  {Univ Paris Diderot, Sorbonne Paris Cit\'e, IMJ-PRG, UMR 7586, F-75013, Paris, France \\
  Sorbonne Universit\'es, UPMC Paris 06, UMR 7586, IMJ-PRG, F-75005, Paris, France \\
  CNRS, UMR 7586, IMJ-PRG, F-75005, Paris, France \\
\em E-mail address: \tt hua.wang@imj-prg.fr}

\end{document}